\documentclass[a4paper]{article}
\usepackage{amsmath,amsthm,amssymb,amscd,ascmac, amsfonts}
\usepackage{mathrsfs}
\usepackage{stmaryrd}
\usepackage{braket}
\usepackage{color}
\usepackage{accents}
\usepackage{url}
\usepackage[margin=20truemm]{geometry}
\allowdisplaybreaks[4]

\newcommand{\Zz}{\mathbb{Z}}
\newcommand{\Cz}{\mathbb{C}}
\newcommand{\Rz}{\mathbb{R}}
\newcommand{\Qz}{\mathbb{Q}}
\newcommand{\hyper}[4]{\left(\begin{array}{c} #1 \\ #2 \end{array}; #3, #4\right)}
\newcommand{\sgn}{{\rm sgn}}

\numberwithin{equation}{section}
\theoremstyle{plain}
\newtheorem{thm}{Theorem}
\newtheorem{defi}{Definition}
\newtheorem{prop}{Proposition}
\newtheorem{lem}{Lemma}
\newtheorem{coro}{Corollary}

\newtheorem{rmk}{Remark}

\begin{document}
\title{A generalization of Zwegers' multivariable $\mu$-function}
% \title{A generalization of Zwegers's $\mu$-function associated with the $q$-Hermite polynomial}
\author{G. Shibukawa and S. Tsuchimi}
\date{}

\maketitle

\begin{abstract}
We introduce a one parameter deformation of Zwegers' multivariable $\mu$-function by applying iterations of the $q$-Borel summation method, which is also a multivariate analogue of the generalized $\mu$-function introduced by the authors. 
For this deformed multivariable $\mu$-function, we give  some formulas, for example, forward shift formula, translation and $\mathfrak{S}_{N+1}$-symmetry. 
Further we mention modular formulas for the Zwegers' original multivariable $\mu$-function. 

%This deformed multivariable $\mu$-function is also a multivariate analogue of the generalized $\mu$-function introduced by the authors. 
%as the image of $q$-Borel and $q$-Laplace transformations of a fundamental solution for a higher order $q$-difference equation. 
% the q-Hermite\UTF{2013} Weber equation. We further give some formulas for our generalized μ-function, for example, forward and backward shift, translation, symmetry, a difference equation for the new pa- rameter, and bilateral q-hypergeometric expressions. From one point of view, the continuous q-Hermite polynomials are some special cases of our μ-function, and the Zwegers’ μ-function is regarded as a continuous q-Hermite polynomial of “−1 degree”.
\end{abstract}

\section{Introduction}
Throughout this paper, let $\tau\in\Cz$ be a complex number with ${\rm Im}(\tau)>0$, $q:=e^{2\pi i\tau}$, $\zeta_N:=e^{\frac{2\pi i}{N}}$ and $\delta_{jk}$ be the Kronecker delta $\delta_{jk}:=\begin{cases}0 &j\neq k\\ 1& j=k\end{cases}$. 
Further we put $a:=q^\alpha, x_j:=e^{2\pi iu_j}, X_N:=x_0\cdots x_N$ and 
\begin{align*}
\widehat{S}:=\left[1+\delta_{jk}\right]_{j,k=1}^{N-1}
  =
  \begin{bmatrix}
  2 & 1 & \cdots & 1 \\
  1 & 2 & \cdots & 1 \\
   \vdots  &  \vdots  & \ddots & \vdots   \\
  1 & 1 & \cdots & 2
  \end{bmatrix}.
\end{align*} 
We define the $q$-shifted factorials, theta functions and Dedekind eta function as follows:
\begin{align*}
    (x)_\infty&=(x;q)_\infty:=\prod_{j=0}^\infty(1-x q^j),\quad (x)_\alpha=(x;q)_\alpha:=\frac{(x;q)_\infty}{(q^\alpha x;q)_\infty}, \ (\alpha\in\Cz), \\
    \vartheta(u)&=\vartheta(u;\tau):=\sum_{\nu\in\Zz+\frac{1}{2}}e^{2\pi i\nu(u+\frac{1}{2})}q^\frac{\nu^2}{2}=-ie^{-\pi iu}q^\frac{1}{8}(q,e^{2\pi iu},e^{-2\pi iu}q)_\infty, \\
    \theta_q(x)&:=\sum_{n\in\Zz}x^nq^\frac{n(n-1)}{2}=(q,-x,-q/x)_\infty, \quad \eta(\tau):=q^\frac{1}{24}(q)_\infty, \\
    \theta_S({\bf{u}};\tau)&=\theta_S(u_1,\ldots,u_{N};\tau):=\sum_{n\in\Zz^{N}}\exp(\pi i{}^t\! nS n\tau+2\pi i({\bf{u}},n)), \\
    \widetilde{\theta}_S({\bf{u}};\tau)&=\widetilde{\theta}_S(u_1,\ldots,u_N;\tau):=\sum_{n\in\Zz^{N}}\exp\left(\pi i{}^t\! nS^{-1} n\tau+2\pi i{}^t\! {\bf{u}}S^{-1}n\right), 
\end{align*}
where $S$ be an $N\times N$ positive definite symmetric matrix and ${\bf{u}}:=[u_j]_{j=1}^N\in\Cz^N$. 
Also, for appropriate complex numbers $a_1,\ldots,a_r,b_1,\ldots,b_s,x$, we define the $q$-hypergeometric series by
\begin{align*}
{}_r\phi_s\hyper{a_1,\ldots,a_r}{b_1,\ldots,b_s}{q}{x}={}_r\phi_s\hyper{\{a_j\}_{1\leq j\leq r}}{\{b_k\}_{1\leq k\leq s}}{q}{x}:=\sum_{n=0}^\infty\frac{(\{a_j\}_{1\leq j\leq r})_n}{(\{b_k\}_{1\leq k\leq s})_n}\left((-1)^nq^\frac{n(n-1)}{2}\right)^{s-r+1}\frac{x^n}{(q)_n}, 
\end{align*}
where $$(a_1,\ldots,a_r)_\alpha=(\{a_j\}_{1\leq j\leq r})_\alpha:=(a_r)_\alpha\cdots(a_r)_\alpha,\quad \alpha\in\Cz\cup\{\infty\}. $$
Finally, for an appropriate function $g(x)$, we define the $q$-shift operator $T_x$ by $T_xg(x):=g(xq)$. 

It is well known that theta functions and the eta function satisfy the following relations (for example, see \cite{M}, \cite{S}): 
\begin{align}
\label{theta relation 1}
  \vartheta(u+1)
  &=-\vartheta(u), \\
\label{theta relation 2} 
  \vartheta(u+\tau)&
  =-e^{-\pi i\tau-2\pi iu}\vartheta(u), \\
\label{theta relation 3} 
  \vartheta(u;\tau+1)
  &=e^\frac{\pi i}{4}\vartheta(u;\tau), \\
\label{theta relation 4}
  \vartheta\left(\frac{u}{\tau};-\frac{1}{\tau}\right)
  &=-i\sqrt{-i\tau}e^\frac{\pi iu^2}{\tau}\vartheta(u;\tau), \\
\label{theta relation 5} 
  \theta_q(x)
  &=x\theta_q(x q), \\
\label{theta relation 6} 
  \theta_q(x)
  &=x\theta_q(x^{-1}), \\
\label{eta relation 1}
  \eta(\tau+1)
  &=e^\frac{\pi i}{12}\eta(\tau), \\
\label{eta relation 2}
  \eta\left(-\frac{1}{\tau}\right)
  &=\sqrt{-i\tau}\eta(\tau), 
  \\
%\label{mul thata relation 1}
%  \theta_S(u_1+1,u_2,\ldots,u_N;\tau)
%  &= 
%  \theta_S(u_1,\ldots,u_N;\tau)
%  \\ 
%\label{mul theta relation 2}
%  \theta_S(u_1+\tau,u_2,\ldots,u_N;\tau)
%  &=
%  \\
%\label{mul theta relation 3}
%  \theta_S(u_1,\ldots,u_N;\tau+1)
%  &=
%  \\
\label{mul theta relation 4}
  \theta_S\left(\frac{u_1}{\tau},\ldots,\frac{u_N}{\tau};-\frac{1}{\tau}\right)
  &=
  \frac{(-i\tau)^\frac{N}{2}e^{\frac{\pi i}{\tau} {}^t\! {\bf{u}} S^{-1} {\bf{u}}}}{\sqrt{\det S}}
  \widetilde{\theta}_S(u_1,\ldots,u_N;\tau)
\end{align}
%
%Mock theta functions first appeared in Ramanujan's last letter to Hardy in $1920$. 
%In this latter, Ramanujan told Hardy that he had discovered a new class of functions which he called mock theta functions. 
%Mock theta functions are functions that have asymptotic behavior similar to ``theta functions'' (i.e., modular forms) at roots of unity but are not ``theta functions'', and Ramanujan gave 17 examples of mock theta functions. 
%Some typical examples are as follows: 
%\begin{align*}
%f_0(q)=\sum_{n=0}^\infty\frac{q^{n^2}}{(-q)_n},\quad \phi(q)=\sum_{n=0}^\infty\frac{q^{n^2}}{(-q^2;q^2)_n},\quad \psi(q)=\sum_{n=1}^\infty\frac{q^{n^2}}{(q;q^2)_n}. 
%\end{align*}
%Later, Andrews and Hickerson gave a detailed definition of the mock theta funcrtion \cite{AH}. 
%For more background on mock theta functions, see, for examples, \cite{AB}, \cite{GM}

%The study of the mock theta functions was greatly advanced by \cite{Zw}. 
%Zwegers introduced a function called the $\mu$-function in Chapter 1 of \cite{Zw}. 
%The definition of the $\mu$-function is as follows: 
Zwegers introduced the following function 
($\mu$-function), 
%which we call ``$\mu$-function'', 
in Chapter 1 of \cite{Zw}:
\begin{align*}
  \mu(u,v;\tau)
  =
  \frac{e^{\pi iu}}{\vartheta(v)}
  \sum_{n\in\Zz}
      \frac{(-1)^ne^{2\pi inv}q^\frac{n(n+1)}{2}}{1-e^{2\pi iu}q^n}.
\end{align*}
The $\mu$-function is a generalization of Ramanujan's mock theta functions, so it is an important function for the study of mock theta functions (see \cite{BFOR}). 

For the $\mu$-function, various interesting formulas hold such as the following:
\begin{align}
\label{mu relation 1}
    \mu(u+1,v)&=\mu(u,v+1)=-\mu(u,v),
    \\ 
\label{mu relation 2}
    \mu(u+\tau,v)&=-e^{2\pi i(u-v)}q^\frac{1}{2}\mu(u,v)-ie^{\pi i(u-v)}q^\frac{3}{8},
    \\
\label{mu relation 3}
    \mu(u+z,v+z)&=\mu(u,v)+\frac{i\eta(\tau)^3\vartheta(u+v+z)\vartheta(z)}{\vartheta(u+z)\vartheta(v+z)\vartheta(u)\vartheta(v)},
    \\
\label{mu relation 4}
    \mu ( u , v )
    &=
    \mu(u + \tau , v + \tau ) 
    = 
    \mu ( -u , -v ) 
    = 
    \mu ( v , u ),
    \\
\label{mu relation 5}
    \mu (u,v;\tau+1)
    &=
    e^{-\frac{\pi i}{4}}
    \mu ( u ,v;\tau),
    \\
\label{mu relation 6}
    \frac{e^{\pi i(u-v)^2/\tau}}{\sqrt{-i\tau}}\mu\left(\frac{u}{\tau},\frac{v}{\tau};-\frac{1}{\tau}\right)
    &=
    -\mu(u,v;\tau)+\frac{1}{2i}h(u-v;\tau),
\end{align}
where the function $h(u;\tau)$ is defined by the following Mordell integral: 
\begin{align}
\label{defi: h}
h(u;\tau)&:=\int_{-\infty}^\infty\frac{e^{\pi ix^2\tau-2\pi xu}}{\cosh(\pi x)}dx. 
\end{align} 
%and is called ``Mordell integral''. 
%and satisfies the following properties: 
%\begin{align}
%\label{eq: h relation 1}
%h(u;\tau)+h(u+1;\tau)
%  &=\frac{2}{\sqrt{-i\tau}}e^{\frac{\pi i}{\tau}(u+\frac{1}{2})^2}, \\
%\label{eq: h relation 2}
%h(u;\tau)+e^{-2\pi iu}q^{-\frac{1}{2}}h(u+\tau;\tau)
%  &=2e^{-\pi iu}q^{-\frac{1}{8}}. 
%\end{align}
%Moreover, $h$ is the unique holomorphic function which satisfies \eqref{eq: h relation 1} and \eqref{eq: h relation 2}. 
Furthermore, the $\mu$-function satisfies a transformation law like Jacobi forms by adding an appropriate non-holomorphic function to the $\mu$-function (modular completion): 
\begin{align}
  \widetilde{\mu}(u+1,v;\tau)&=\widetilde{\mu}(u,v+1;\tau)=-\widetilde{\mu}(u,v;\tau), 
  \\
  \widetilde{\mu}(u,v;\tau)
  &=
  -e^{-2\pi i(u-v)}q^{-\frac{1}{2}}\widetilde{\mu}(u+\tau,v;\tau)=-e^{-2\pi i(v-u)}q^{-\frac{1}{2}}\widetilde{\mu}(u,v+\tau:\tau), 
    \\
  \widetilde{\mu}(u,v;\tau)
  &=
  e^\frac{\pi i}{4}\widetilde{\mu}(u,v;\tau+1), 
    \\
  \widetilde{\mu}(u,v;\tau)
  &=-\frac{e^{\frac{\pi i(u-v)^2}{\tau}}}{\sqrt{-i\tau}}\widetilde{\mu}\left(\frac{u}{\tau},\frac{v}{\tau};-\frac{1}{\tau}\right), 
\end{align}
where 
\begin{align}
E(x)&:=2\int_0^x e^{-\pi z^2}dz,
\\
\label{defi: R}
R(u;\tau)&:=\sum_{\nu\in\Zz+\frac{1}{2}}\left\{{\rm sgn}(\nu)-E((\nu+a)\sqrt{2t})\right\}(-1)^{\nu-\frac{1}{2}}e^{-2\pi i\nu u}q^{-\frac{\nu^2}{2}},
\\
\widetilde{\mu}(u,v;\tau)&:=\mu(u,v;\tau)+\frac{i}{2}R(u-v;\tau), 
\end{align}
with $t={\rm Im}(\tau),a=\frac{{\rm Im}(u)}{{\rm Im}(\tau)}$. 
%These results were pioneering works in the study of mock modular forms (see \cite{BFOR}). 
%Furthermore, in view of the character theory, Zwegers introduced a multivariate analogue of the $\mu$-function in \cite{Zw2}. 
%\begin{defi}
%For $u_0,\ldots,u_N\not\in\Lambda_\tau$, we define 
%\begin{align*}
%\mu_N(u_0,\ldots,u_N;\tau):=\sum_{n\in\Zz^N}\frac{e^{\pi iu_0}}{1-e^{2\pi iu_0}q^{|n|}}\prod_{j=1}^N\frac{(-1)^{n_j}e^{-2\pi in_ju_j}q^\frac{n_j(n_j+1)}{2}}{\vartheta(u_j;\tau)},
%\end{align*}
%where $|n|:=n_1+\cdots+n_N$. We call this function ``multivariable $\mu$-function''.
%\end{defi}

Recently, there are some studies of the $\mu$-function from the view of certain $q$-Borel summation method such as \cite{GW}, \cite{ST}. 
The $q$-Borel summation method is one of the techniques to construct convergent series solutions from divergent series solutions of $q$-difference equations. 
Namely, for a formal series $g(x)=\sum_{n=0}^\infty a_n x^n$, we define the $q$-Borel transformation $\mathcal{B}_q$ and the $q$-Laplace transformation $\mathcal{L}_q$ as follows: 
\begin{align}
\label{q-Borel Laplace transformation}
\mathcal{B}_q(g)(\xi):=\sum_{n=0}^\infty a_nq^\frac{n(n-1)}{2}\xi^n,\quad \mathcal{L}_q(g)(x;\lambda):=\sum_{n\in\Zz}\frac{g(\lambda q^n)}{\theta_q(\lambda q^n/x)}. 
\end{align}
The $q$-Laplace transformation has the following Jackson integral expression: 
\begin{align*}
\mathcal{ L }_q(g)(x,\lambda)
   =
   \frac{1}{1-q}\int_0^{\lambda \infty}\frac{g(t)}{\theta_q(t/x)}\ dt_q, \quad\int_0^{\lambda\infty} f(t)\ dt_q:=(1-q)\sum_{n\in\Zz}^\infty f(\lambda q^n). 
\end{align*}
Hence, the extra parameter $\lambda$ of $\mathcal{L}_q$ seems to be the argument parameter of the Jackson integral path. 

From a simple calculation, for $n\in\Zz$ and $\varphi_n(x):=x^n$, we have 
\begin{align}
\label{mono summation}
\mathcal{L}_q\circ\mathcal{B}_q(\varphi_n)(x,\lambda)=x^n. 
\end{align}
We remark that the extra parameter $\lambda$ of $\mathcal{L}_q$ does not appear in the right hand side of \eqref{mono summation}. 
Then, we see that the image of the composition of the $q$-Borel and $q$-Laplace transformations of a convergent series solution of a $q$-difference equation (operator) $D\in\Cz[x][T_x]$ is equal to the original convergent series solution. 
%From a simple calculation, since $\mathcal{L}_q\circ\mathcal{B}_q(x^n)=x^n$ for $n\in\Zz$, we see that the image of the composition of the $q$-Borel and $q$-Laplace transformations of a convergent series solution of a $q$-difference equation is equal to the original convergent series solution. 

Also, for a divergent series solution $g(x)$ of
%a $q$-difference equation
$D$, by the following operator method
%\begin{align*}
%\mathcal{L}_q\circ\mathcal{B}_q(x^mT_x^n[g])(x,\lambda)=x^mT_x^n[\mathcal{L}_q\circ\mathcal{B}_q(g)(x,\lambda)], \quad T_xg(x):=g(xq), 
%\end{align*}
\begin{align*}
  \mathcal{B}_q(x^mT_x^n[g])(\xi)
  &=
  q^\frac{m(m-1)}{2}\xi^mT_\xi^{m+n}\mathcal{B}_q([g])(\xi),
  \\
  \mathcal{L}_q(\xi^mT_\xi^n[g])(x,\lambda)
  &=
  q^{-\frac{m(m-1)}{2}}x^mT_x^{n-m}\mathcal{L}_q([g])(x,\lambda), 
  \\
  \mathcal{L}_q\circ\mathcal{B}_q(x^mT_x^n[g])(x,\lambda)
  &=
  x^mT_x^n[\mathcal{L}_q\circ\mathcal{B}_q(g)(x,\lambda)], 
\end{align*}
the function $\mathcal{L}_q\circ\mathcal{B}_q(g)(x,\lambda)$ is a solution of 
%the $q$-difference equation 
$\mathcal{L}_q\circ\mathcal{B}_q(D)=D$. 
Therefore, if the image of the composition $\mathcal{L}_q\circ\mathcal{B}_q(g)(x,\lambda)$ converges, we obtain a convergent series solution of $D$. 
%the original $q$-difference equation. 
We remark that the convergent solution $\mathcal{L}_q\circ\mathcal{B}_q(g)(x,\lambda)$ is deformed by the extra parameter $\lambda$ of $\mathcal{L}_q$ which does not appear in the original $q$-difference operator $D$. 

Here, we consider an important example for the $q$-difference equation: 
%The above story is one of the overview of the $q$-Borel summation methods. 
%Then, taking the $q$-difference equation 
\begin{align}
\label{factorized eq}
[T_x^2-(1-xq)T_x-x]f(x)=[T_x-1][T_x+x]f(x)=0. 
\end{align} 
%as an example, 
%The image of the composition of the $q$-Borel and $q$-Laplace transformations of 
In this case, a divergent series solution of \eqref{factorized eq} around $x=0$ is 
\begin{align*}
g_0(x):=\sum_{n=0}^\infty(-x)^nq^{-\frac{n(n+1)}{2}}, 
\end{align*} 
and the image of the composition of the $q$-Borel and $q$-Laplace transformations of $g_0(x)$ is
\begin{align*}
\mathcal{L}_q\circ\mathcal{B}_q(g_0)(x,\lambda)=\frac{1}{\theta_q(x/\lambda)}\sum_{n\in\Zz}\frac{(\lambda/x)^nq^\frac{n(n+1)}{2}}{1+\lambda q^n}. 
\end{align*}
%Namely, 
Then we show that the $\mu$-function is derived from the $q$-Borel summation method: 
% from the following formula: 
\begin{align*}
\mu(u,v;\tau)=-ie^{\pi i(u-v)}q^{-\frac{1}{8}}\mathcal{L}_q\circ\mathcal{B}_q(g_0)(e^{2\pi i(u-v)},-e^{2\pi iu}). 
\end{align*}

Furthermore, in \cite{ST}, the authors introduced a generalization of the Zwegers' $\mu$-function as a fundamental solution of the $q$-Hermite-Weber equation: 
%In \cite{ST}, the authors studied the following $q$-difference equation: 
 \begin{align}
 \label{eq: mua equation}
 [T_x^2-(1-xq)\sqrt{a}T_x-xq]f(x)=0,\quad T_xf(x):=f(qx). 
 \end{align}
 The $q$-difference equation \eqref{eq: mua equation} is a one parameter deformation of the $q$-difference equation \eqref{factorized eq} (the case of $a=q$) and its divergent series solution around $x=0$: 
 \begin{align}
\label{func: mua div sol}
x^\frac{\alpha}{2}\tilde{f}_0(a,x):=x^\frac{\alpha}{2}{}_2\phi_0\hyper{a,0}{-}{q}{\frac{x}{a}}, \quad a:=q^\alpha. 
\end{align} 
%Noting that the $\mu$-function is a solution of the $q$-difference equation \eqref{eq: mua equation} with $a=q$ from the equation \eqref{mu relation 2}. 
%If we determine the coefficients and characteristic exponents of a formal series solution around $x=0$ so that there are solutions of the $q$-difference equation \eqref{eq: mua equation}, the $q$-difference equation has the following divergent solution: 
%A formal solution around $x=0$ of \eqref{eq: mua equation} is the following:
%\begin{align}
%\label{func: mua div sol}
%x^\frac{\alpha}{2}\tilde{f}_0(a,x):=x^\frac{\alpha}{2}{}_2\phi_0\hyper{a,0}{-}{q}{\frac{x}{a}}. 
%\end{align} 
%By using the $q$-Borel summation of this divergent solution \eqref{func: mua div sol}, 
%The authors introduce a generalization of the $\mu$-function as the image of the composition of the $q$-Borel and Laplace transformations of this divergent solution \eqref{func: mua div sol}.  
%\begin{defi}[generalized $\mu$-function]
For $u-\alpha\tau,v\not\in\Lambda_\tau$, we define a generalization of the $\mu$-function (generalized $\mu$-function) as follows: 
 \begin{align*}
\mu(u,v;\alpha;\tau):=\frac{e^{\pi i\alpha(u-v)}}{\vartheta(v)}\sum_{n\in\Zz}(-1)^ne^{2\pi i(n+\frac{1}{2})v}q^\frac{n(n+1)}{2}\frac{(e^{2\pi iu}q^{n+1};q)_\infty}{(e^{2\pi iu}q^{n-\alpha+1};q)_\infty}. 
\end{align*}
In particular, we have $\mu(u,v;0;\tau)=-iq^{-\frac{1}{8}}$ and $\mu(u,v;1;\tau)=\mu(u,v;\tau)$. 
%We call this function ``generalized $\mu$-function''. 
%\end{defi}
%
%From the equation \eqref{mu relation 2}, the $\mu$-function satisfies the following $q$-difference equation: 
%\begin{align}
%\label{mu difference equation}
%\left(T_x-1\right)\left(T_x+x\right)f(x)=0,\quad T_xf(x):=f(qx). 
%\end{align}
%This $q$-difference equation \eqref{mu difference equation} has a divergent solution around $x=0$ such as the following: 
%\begin{align}
%\label{mu eq divergent solution}
%\widetilde{f}_0(x):=\sum_{n=0}^\infty(-1)^nq^{-\frac{n(n+1)}{2}}x^n. 
%\end{align}
%In fact, it is known that the $\mu$-function essentially coincides with the image of the composition of the $q$-Borel transformation and the $q$-Laplace transformation of this divergent solution \eqref{mu eq divergent solution} \cite[Theorem 1.2]{ST}. 
%Namely, putting $f_0(a,x,\lambda):=x^\frac{\alpha}{2}\mathcal{L}_q\circ\mathcal{B}_q(\widetilde{f}_0(a,x))(x;\lambda)$, 

The authors proved that the generalized $\mu$-function is also written by the image of $\mathcal{L}_q\circ\mathcal{B}_q$ via $\tilde{f}_0$: 
\begin{align*}
\mu(u,v;\alpha;\tau)=-iq^{-\frac{1}{8}}f_0(q^\alpha,e^{2\pi i(u-v)},-e^{2\pi iu}), 
\end{align*}
where $f_0(a,x,\lambda):=x^\frac{\alpha}{2}\mathcal{L}_q\circ\mathcal{B}_q(\widetilde{f}_0(a,x))(x;\lambda)$. 

The above description of the $\mu$-function by the $q$-Borel summation method has some advantages. 
First, from the analytical viewpoints of the $q$-difference equations \eqref{factorized eq}, we treat various formulas 
%such as transformation formulas 
for the $\mu$-function (expect some modular properties \eqref{mu relation 5} and \eqref{mu relation 6}) clearly. 
For example, the translation formula \eqref{mu relation 3} is regarded as a connection formula of solutions for \eqref{factorized eq}. 
%and connection formulas of the $q$-hypergeometric series, we see that 
%{\color{red}{???????????????????????????. }}

Second, we give an interpretation why the $\mu$-function which is a solution of one variable $q$-difference equation \eqref{factorized eq} has two variables. 
Namely, by inserting the extra parameter $\lambda$ of the $q$-Laplace transformation $\mathcal{L}_q$, the divergent series solution $g_0(x)$ of the $q$-difference equation \eqref{factorized eq} is naturally deformed to the Zwegers' $\mu$-function which is a two variable function. 
%Here the definition of the $q$-Borel transformation $\mathcal{B}_q$ and the $q$-Laplace transformation $\mathcal{L}_q$ introduced by \cite{RSZ} are 
%\begin{align}
%\label{q-Borel Laplace transformation}
%\mathcal{B}_q(g)(\xi):=\sum_{n=0}^\infty a_nq^\frac{n(n-1)}{2}\xi^n,\quad \mathcal{L}_q(g)(x;\lambda):=\sum_{n\in\Zz}\frac{g(\lambda q^n)}{\theta_q(\lambda q^n/x)}, 
%\end{align}
%for a formal series $g(x)=\sum_{n=0}^\infty a_n x^n$. 
%Since $\mu(u,v;1;\tau)=\mu(u,v;\tau)$, we see that the original $\mu$-function is derived from the $q$-Borel summation method of the divergent solution: 
%\begin{align*}
%x^\frac{1}{2}{}_2\phi_0\hyper{q,0}{-}{q}{\frac{x}{q}}
%%=\sum_{n=0}^\infty q^{-\frac{n(n+1)}{2}}x^{n+\frac{1}{2}}
%\end{align*} of the $q$-difference equation
%\begin{align*}
%[T_x^2-(1-xq)q^\frac{1}{2}T_x-xq]f(x)=[T_x-q^\frac{1}{2}][T_x+xq^\frac{1}{2}]f(x)=0
%\end{align*}. 

On the other hand, in view of the character theory related to super affine Lie algebras, Zwegers \cite{Zw2} introduced a multivariate analogue of the $\mu$-function. 
\begin{defi}
For $u_0,\ldots,u_N\not\in\Lambda_\tau$, we define 
\begin{align*}
\mu_N(u_0,\ldots,u_N;\tau):=\sum_{n\in\Zz^N}\frac{e^{\pi iu_0}}{1-e^{2\pi iu_0}q^{|n|}}\prod_{j=1}^N\frac{(-1)^{n_j}e^{-2\pi in_ju_j}q^\frac{n_j(n_j+1)}{2}}{\vartheta(u_j;\tau)},
\end{align*}
where $|n|:=n_1+\cdots+n_N$. We call this function ``multivariable $\mu$-function''.
\end{defi}

The purpose of this paper is to introduce a multivariate analogue of the generalized $\mu$-function by applying iterations of the $q$-Borel summation method and to present an analogue of the formulas \eqref{mu relation 1}--\eqref{mu relation 4}. 
%We define 
%the multivariate analogue of the generalized $\mu$-function as follows. 
\begin{defi}
For $u_0,\ldots,u_N\not\in\Lambda_\tau$, we define 
\begin{align*}
\hat{\mu}_N=\hat{\mu}_N(u_0,\ldots,u_N;\alpha;\tau):=e^{\pi i(\alpha-1)u}\sum_{n\in\Zz^N}\frac{(e^{\pi iu_0+\alpha\tau}q^{|n|})_\infty}{e^{-\pi iu_0}(e^{2\pi iu_0}q^{|n|})_\infty}\prod_{j=1}^N\frac{(-1)^{n_j}e^{-2\pi in_ju_j}q^\frac{n_j(n_j+1)}{2}}{\vartheta(u_j;\tau)}, 
\end{align*}
where $u:=u_0+\cdots+u_N$. 
We call this function ``multivariable generalized $\mu$-function''.
\end{defi}

This multivariable generalized $\mu$-function is a multivariate analogue of the generalized $\mu$-function: 
\begin{align}
\label{eq: mul mua and mua}
-e^{\pi i\alpha(\alpha-1)\tau}\hat{\mu}_1(u_0-\alpha\tau,-v;\alpha;\tau)
  &=
  \mu(u_0,u_1;\alpha;\tau), 
\end{align}
and a one parameter deformation of the Zwegers' multivariable $\mu$-function: 
\begin{align}
\label{eq: mul mua and mul mu}
\hat{\mu}_N(u_0,\ldots,u_N;1;\tau)&=\mu_N(u_0,\ldots,u_N;\tau). 
\end{align}
%The relations between the multivariable generalized $\mu$-function and the previous $\mu$-functions are as follows:
%\begin{align}
%-\hat{\mu}_1(u_0,-u_1;1;\tau)&=\mu(u_0,u_1), 
%\\
%\label{eq: mul mua and mul mu}
%\hat{\mu}_N(u_0,\ldots,u_N;1;\tau)&=\mu_N(u_0,\ldots,u_N;\tau), 
%\\
%\label{eq: mul mua and mua}
%-e^{\pi i\alpha(\alpha-1)\tau}\hat{\mu}_1(u_0-\alpha\tau,-v;\alpha;\tau)&=\mu(u_0,u_1;\alpha;\tau). 
%\end{align}

The contents of this paper are that we derive the multivariable generalized $\mu$-function from the image of the composition of the $N$-th order $q$-Borel transformation $\mathcal{B}_q^N$ and the $N$-th order $q$-Laplace transformation $\mathcal{L}_q^N$ of a divergent solution of a certain $q$-difference equation (see Theorem \ref{thm: fN=muN}), and show the following various formulas of the multivariable generalized $\mu$-function $\hat{\mu}_N$. 
%{\color{red}{多変数$\mu$函数はある$q$差分方程式に現れる発散級数解の$q$-Borel $q$-Laplace変換の合成の像を種にして構成したことを書く. }}

%This multivariable $\mu$-function essentially coincides with the multivariable Appell function introduced by \cite{KW}. 
%The definition of the multivariable Appell function is as follows: 
%\begin{align}
%\label{defi: multivariable Appell function}
%A_{B,\mathcal{L}}(w;x_1,\ldots,x_N;q):=\sum_{n\in\Zz^{N}}\frac{q^{\frac{1}{2}{}^t\! nBn}x_1^{n_1}\cdots x_N^{n_N}}{1+wq^{\mathcal{L}(n)}}, 
%\end{align}
%where $B$ be a positive definite symmetric matrix and $\mathcal{L}$ be a linear function on $\Cz^N$. Namely, we assume $B=[\delta_{jk}l]_{j,k=1}^N, \mathcal{L}(n)=|n|$, then we have
%\begin{align}
%&A_{B,\mathcal{L}}(-x_0^\frac{1}{l},-\frac{1}{x_1},\ldots,-\frac{1}{x_N};q)\nonumber\\
%\label{multivariable Appell function and multivariable mu function}
%&=\sum_{k=0}^{l-1}e^\frac{\pi i(2k-l)u_0}{l}q^{-\frac{(2k-l)N}{2}}\mu_N(u_0,u_1+k\tau,\ldots,u_N+k\tau;l\tau)\prod_{j=1}^N\vartheta(u_j+(l-k)\tau;l\tau).
%\end{align} 
%There are many studies on connections between the multivariable Appell function and mock modular forms including \cite{Zw} (for examples, see \cite{BF}, \cite{BO} and \cite{Zw2}).
%In this paper, we present modular transformations and a modular completion of the multivariable Appell function in the case \eqref{multivariable Appell function and multivariable mu function}. 

%First, we give the following formulas satisfied by the multivariable generalized $\mu$-function similar to those of the original $\mu$-function. 
\begin{thm}
\label{thm: mul mua}
%Put $a:=q^\alpha, x_j:=e^{2\pi iu_j}, X_N:=x_0\cdots x_N, \widehat{S}:=\left[1+\delta_{jk}\right]_{j,k=1}^{N-1}$. 
For any $r,s\in\{0,\ldots,N\}$ and $\sigma \in \mathfrak{S}_{N+1}$, 
% is an element of the symmetric group $\mathfrak{S}_{N+1}$
then we have
%$\hat{\mu}_N$ satisfies the following formulas. 
\begin{align}
\label{mul mua relation 1}
    \hat{\mu}_N(\ldots,u_r+\tau,\ldots;\alpha;\tau)
    &=
    \hat{\mu}_N(\ldots,u_s+\tau,\ldots;\alpha;\tau), 
    \\
\label{mul mua relation 2}
    [(a-1)T_a-\sqrt{aX_N}T_{x_0}+\sqrt{X_N}]\hat{\mu}_N
    &=
    0,
    \\
\label{mul mua relation 3}
   \left[T_{x_0}^{N+1}-\sqrt{a}T_{x_0}^N+\sqrt{a^{N+2}}X_NT_{x_0}-\sqrt{a^{N+1}}\right]\hat{\mu}_N
   &=
   0,
   \\
\label{mul mua relation 4}   
   \hat{\mu}_N-\hat{\mu}_N(u_0-z,u_1+z,u_2,\ldots,u_N;\alpha;\tau)
   &=
   \frac{e^{\pi i(\alpha-1)u+\pi i\alpha\tau}\vartheta(\alpha\tau)\vartheta(z)\vartheta(z-u_0+u_1)(q)_\infty}{\vartheta(z-u_0)\vartheta(z+u_1)\vartheta(u_0)\cdots\vartheta(u_N)(q^{1-\alpha})_\infty}\nonumber
   \\
   &\quad\times\sum_{n\in\Zz^{N-1}}q^\frac{{}^t\! n\widehat{S}n}{2}
   \prod_{j=1}^{N-1}e^{2\pi in_j(u_0+u_1-u_{j+1})}\nonumber
    \\
    &\quad\times
    %\sum_{m=0}^\infty\frac{(q^{m+1})_\infty}{(q^{m+1-\alpha})_\infty}(-1)^mq^{\frac{m(m-1)}{2}+m|n|}e^{2\pi im(u_0+u_1+\alpha\tau)},
    {}_1\phi_1\hyper{q^{1-\alpha}}{0}{q}{e^{2\pi i(u_0+u_1+\alpha\tau)}q^{|n|}}, 
    \\
\label{mul mua relation 5}
    \hat{\mu}_N(u_0,\ldots,u_N;\alpha;\tau)
    &=
    \hat{\mu}_N(u_{\sigma(0)},\ldots,u_{\sigma(N)};\alpha;\tau), 
    \\
\label{mul mua relation 6}
    e^{2\pi i\alpha u}q^\frac{\alpha^2(N+1)}{2}\hat{\mu}_N(-\alpha\tau-u_0,\ldots,-\alpha\tau-u_N;\alpha+1;\tau)
    &=
    \frac{(-1)^{N+1}\vartheta(u_0)\cdots\vartheta(u_N)}{\vartheta(u_0+\alpha\tau)\cdots\vartheta(u_N+\alpha\tau)}
    \nonumber\\
    &\quad \times
    \hat{\mu}_N(u_0,\ldots,u_N;\alpha+1;q), 
\end{align}
where 
%$(r,s)\in\{0,\ldots,N\}^2$ and $\sigma$ is an element of the symmetric group 
$\mathfrak{S}_{N+1}$ is the symmetric group of the degree $N+1$.
%, a:=q^\alpha, x_j:=e^{2\pi iu_j}, X_N:=x_0\cdots x_N, \widehat{S}:=\left[1+\delta_{jk}\right]_{j,k=1}^{N-1}
%\in{\rm Sym}_{N-1}(\Zz), 
%$\sigma\in\mathfrak{S}_{N+1}$.
\end{thm}
In the case of $N=1$, the formulas \eqref{mul mua relation 1}, \eqref{mul mua relation 2}, \eqref{mul mua relation 3}, \eqref{mul mua relation 4}, \eqref{mul mua relation 5} and \eqref{mul mua relation 6} are correspond to (1.28), (1.25), (1.23), (1.27), (1.29) and (1.30) of Theorem 1.3 in \cite{ST}, respectively. 

We remark that from the above equation \eqref{mul mua relation 5}, the function $\hat{\mu}_N$ is not only a multivariate analogue of the generalized $\mu$-function, but also a symmetric function.

As a corollary, from 
%Theorem \ref{thm: mul mua} and 
the equation \eqref{eq: mul mua and mul mu} we obtain the following formulas of the Zwegers' multivariable $\mu$-function, which are not mentioned in \cite{Zw2}. 
%satisfies the following corollary. 
\begin{coro}
\label{prop: muN}
We have
%The multivariable $\mu$-function satisfies the following formulas: 
\begin{align}
\label{mun relation 1}
  \mu_{N}(u_0+1,u_1,\ldots,u_{N};\tau)
  &=
  -\mu_{N}(u_0,u_1,\ldots,u_N;\tau), 
  \\
\label{mun relation 2}
  \mu_{N}(u_0+N\tau,u_1,\ldots,u_N;\tau)
  &=
  (-1)^Ne^{2\pi iu}q^\frac{N}{2}\mu_{N}(u_0,\ldots,u_{N};\tau)+i^Ne^{\pi iu}q^\frac{3N}{8}, 
  \\
  \left[T_{x_0}-q^\frac{1}{2}\right]\left[T_{x_0}^N-(-1)^NX_Nq^\frac{N}{2}\right]\mu_N
  &=0, 
\\
\label{mun relation 3}
  \mu_N(\ldots,u_r+\tau,\ldots,u_s-\tau,\ldots;\tau)
  &=
  \mu_N(u_0,\ldots,u_N;\tau), 
\\
\label{mun relation 4}
  \mu_N(-u_0,\ldots,-u_N;\tau)
  &=(-1)^{N+1}\mu_N(u_0,\ldots,u_N;\tau), 
\\
\label{mun relation 5}
  \mu_N(u_0,\ldots,u_N;\tau)
  &=
  \mu_N(u_{\sigma(0)},\ldots,u_{\sigma(N)};\tau), 
\\
  \mu_N(u_0-z,u_1+z,u_2,\ldots,u_N;\tau)
  &=
  \mu_N(u_0,u_1,u_2,\ldots,u_N;\tau)\nonumber
\\
\label{mun relation 6}
&\quad +\frac{i\eta(\tau)^3\vartheta(z)\vartheta(z-u_0+u_1)}{\vartheta(z-u_0)\vartheta(z+u_1)\vartheta(u_0)\cdots\vartheta(u_N)}\theta_{\widehat{S}}(v_1,\ldots,v_{N-1};\tau)
%\\
%\label{mu relation 7}
%&(7)\ \vartheta(u_0)\cdots\vartheta(u_N)\mu_N(u_0,\ldots,u_N;\tau)=\sum_{n\in\Zz^{N+1}}\sum_{k=0}^\infty(-1)^kq^{k|n|+\frac{k(k+1)}{2}}\prod_{j=0}^{N}(-x_j)^{n_j}q^\frac{n_j(n_j+1)}{2}
\end{align}
where $v_j:=u_0+u_1-u_{j+1}, v:={}^t\![v_1,\ldots ,v_{N-1}]\in\Cz^{N-1}$. 
\end{coro}
%We see that these formulas correspond to the properties satisfied by the original $\mu$-function as \eqref{mun relation 1}$\leftrightarrow$\eqref{mu relation 1}, \eqref{mun relation 2}$\leftrightarrow$\eqref{mu relation 2}, \eqref{mun relation 3}$\leftrightarrow$the first equal of \eqref{mu relation 4}, \eqref{mun relation 4}$\leftrightarrow$ the second equal of \eqref{mu relation 4}, \eqref{mun relation 5}$\leftrightarrow$ the third equal of \eqref{mu relation 4} and \eqref{mun relation 6}$\leftrightarrow$\eqref{mu relation 3}.

Next, we construct a vector valued function $M_N$ whose components are the multivariable $\mu$-functions, and show some formulas of the function $M_N$. 
\begin{thm}
\label{prop: MN}
We define
\begin{align}
M_N=M_N(u_0,\ldots,u_N;\tau):=\left[(-1)^ke^{-\frac{2\pi iku}{N}-\frac{\pi ik^2\tau}{N}}\mu_N(u_0+k\tau,u_1,\ldots,u_N;\tau)\right]_{k=0}^{N-1}, 
\end{align}
then we have 
\begin{align}
\label{eq: MN 1}
\left[\delta_{jk}\zeta_N^{k}\right]_{j,k=0}^{N-1}M_N(u_0+1,u_1,\ldots,u_N;\tau)
  &=
  -M_N,
  \\
\label{eq: MN 2}
%&(2)\ M_N(u_0+\tau,u_1,\ldots,u_N;\tau)=-e^\frac{2\pi iu}{N}q^\frac{1}{2N}
%\begin{bmatrix}
%0 & 1 & 0 & \ldots & 0\\
%0 & 0 & 1 & \ldots & 0\\
% & & & \ldots & \\
%0 & 0 & 0 & \ldots & 1\\
%1 & 0 & 0 & \ldots & 0
%\end{bmatrix}
%M_N-
%\begin{bmatrix}
%0 \\ 0 \\ \vdots \\ 0 \\ (-i)^Ne^{-\pi iu+\frac{2\pi iu}{N}}q^{-\frac{N}{8}+\frac{1}{2N}}
%\end{bmatrix},
%\\
e^{-\frac{2\pi iu}{N}}q^{-\frac{1}{2N}}\begin{bmatrix}
0 & \ldots & 0 & 0 & 1\\
1 & \ldots & 0 & 0 & 0\\
 & \ddots & & & \\
 0 & \ldots & 1 & 0 & 0 \\
0 & \ldots & 0 & 1 & 0
\end{bmatrix}
M_N(u_0+\tau,u_1,\ldots, u_N;\tau)
&=-M_N-(-i)^Ne^{-\pi iu}q^{-\frac{N}{8}}
\begin{bmatrix}
1 \\ 0 \\ \vdots \\ 0 \\ 0
\end{bmatrix},
  \\
\label{eq: MN 3}
  M_N(u_0-z,u_1+z,u_2,\ldots,u_N;\tau)-M_N
  &=
  \Phi_N(u_0,\ldots,u_N;z;\tau),
  \\
\label{eq: MN 4}
  M_N(u_0,\ldots,u_N;\tau)
  &=
  M_N(u_{\sigma(0)},\ldots,u_{\sigma(N)};\tau),
  \\
\label{eq: MN 5}
  M_N(,\ldots ,u_r-\tau,\ldots ,u_s+\tau,\ldots ;\tau)
  &=
  M_N(\ldots ,u_r,\ldots, u_s,\ldots ;\tau),
\end{align}
where
\begin{align}
\label{defi: Phi}
\Phi_N(u_0,\ldots,u_N;z;\tau)&:=\frac{i\eta(\tau)^3\vartheta(z)\vartheta(z-u_0+u_1)}{\vartheta(z-u_0)\vartheta(z+u_1)\vartheta(u_0)\cdots\vartheta(u_N)}[\nu_{N,k}(u_0,\ldots,u_N;\tau)]_{k=0}^{N-1},\\
\label{defi: nu}
\nu_{N,k}(u_0,\ldots,u_N;\tau)&:=\sum_{n\in\Zz^{N-1}+e_k}\exp(\pi i{}^t\! n\widehat{S}n\tau+2\pi i(v,n))\quad e_k:={}^t\![k/N,\ldots,k/N]\in\Qz^{N-1}. 
\end{align}
\end{thm}
For the function $M_N$ also satisfies the following modular transformations. 
%We point that Zwegers did not mention these results in \cite{Zw2}. 
\begin{thm}
\label{thm: MN modular}
We have
\begin{align}
\label{eq: MN modular 1}
  M_N(u_0,\ldots ,u_N;\tau)
  &=
  e^{\frac{\pi iN}{4}}\left[\delta_{jk}(-1)^k\zeta_{2N}^{k^2}\right]_{j,k=0}^{N-1}M_N(u_0,\ldots ,u_N;\tau+1)
  \\
  M_N(u_0,\ldots,u_N;\tau)
  &=
  \frac{(-i)^{N+1}e^{\frac{\pi iu^2}{N\tau}}}{N^\frac{1}{2}\sqrt{-i\tau}}\left[\zeta_N^{-jk}\right]_{j,k=0}^{N-1}M_N\left(\frac{u_0}{\tau},\ldots,\frac{u_N}{\tau};-\frac{1}{\tau}\right)
  \nonumber \\
\label{eq: MN modular 2}
  &\quad-
  \frac{(-1)^Ni}{2}\left[(-1)^ke^{-\frac{2\pi iku}{N}}q^{-\frac{k^2}{2N}}h\left(u+k\tau-\frac{N-1}{2};N\tau\right)\right]_{k=0}^{N-1}. 
\end{align}
\end{thm}
%We point that Zwegers did not mention these results in \cite{Zw2}. 
Furthermore, we give a modular completion of the function $M_N$ as follows. 
%This result is already shown by Zwegers in \cite{Zw2}, and we give another proof using Theorem \ref{thm: MN modular} in this paper. 
\begin{thm}[\cite{Zw2}]
\label{thm: modular completion}
We define
\begin{align}
R_N(u;\tau)&:=\left[(-1)^ke^{-\frac{2\pi iku}{N}}q^{-\frac{k^2}{2N}}R\left(u+k\tau+\frac{N+1}{2};N\tau\right)\right]_{k=0}^{N-1}, \\
\label{modular completion of MN}
\widetilde{M}_N(u_0,\ldots,u_N;\tau)&:=M_N(u_0,\ldots,u_N;\tau)+\frac{i}{2}R_N(u;\tau). 
\end{align}
Then, we have 
\begin{align}
\label{eq: MN completion 1}
  -\left[\delta_{jk}\zeta_N^{k}\right]_{j,k=0}^{N-1}\widetilde{M}_N(u_0+1,u_1,\ldots,u_N;\tau)
  &=\widetilde{M}_N(u_0,\ldots,u_N;\tau),\\
\label{eq: MN completion 2}
  -e^{-\frac{2\pi iu}{N}}q^{-\frac{1}{2N}}\begin{bmatrix}
0 & \ldots & 0 & 0 & 1\\
1 & \ldots & 0 & 0 & 0\\
 & \ddots & & & \\
 0 & \ldots & 1 & 0 & 0 \\
0 & \ldots & 0 & 1 & 0
\end{bmatrix}
  \widetilde{M}_N(u_0+\tau,u_1,\ldots,u_N;\tau)
  &=\widetilde{M}_N(u_0,\ldots,u_N;\tau),\\
\label{eq: MN completion 3}
  e^{\frac{\pi iN}{4}}\left[\delta_{jk}(-1)^k\zeta_{2N}^{k^2}\right]_{j,k=0}^{N-1}\widetilde{M}_N(u_0,\ldots,u_N;\tau+1)
  &=\widetilde{M}_N(u_0,\ldots,u_N;\tau),\\
\label{eq: MN completion 4}
  \frac{(-i)^{N+1}e^\frac{\pi iu^2}{N\tau}}{N^\frac{1}{2}\sqrt{-i\tau}}\left[\zeta_N^{-jk}\right]_{j,k=0}^{N-1}\widetilde{M}_N\left(\frac{u_0}{\tau},\dots,\frac{u_N}{\tau};-\frac{1}{\tau}\right)
  &=\widetilde{M}_N(u_0,\ldots,u_N;\tau).
\end{align}
\end{thm}
Theorem \ref{thm: modular completion} is already shown by Zwegers in \cite{Zw2} essentially, but Zwegers did not mention Theorem \ref{thm: MN modular}. 
In this paper, we give another proof of Theorem \ref{thm: modular completion} from Theorem \ref{thm: MN modular}. 
Furthermore, as some applications of Theorem \ref{thm: MN modular} and Theorem \ref{thm: modular completion}, we propose modular translations of the multivariable $\mu$-function.

%{\color{red}{This paper is organized as follows. }}
%First, in Section \ref{sec: 2}, we construct a special solution of a certain factorized higher-order $q$-difference equation to derive the multivariable $\mu$-function, and show basic properties satisfied by this special solution in a more general form. 
%Moreover, we derive the multivariable $\mu$-function based on this special solution and prove Proposition \ref{prop: muN} and Proposition \ref{prop: MN}. 
%Next, in Section \ref{sec: 3}, we present some important results about the paper \cite{Zw} and modular forms to prove our main results. 
%Furthermore, we show Theorem \ref{thm: MN modular} and Theorem \ref{thm: modular completion} which are the main results of this paper. 
%Finally, in Section \ref{sec: 4}, as applications of Theorem \ref{thm: MN modular} and Theorem \ref{thm: modular completion}, we give modular transformations of the multivariable $\mu$-function. 

\section{Proofs of Theorem \ref{thm: mul mua} and \ref{prop: MN}}
In this section, we derive the multivariable generalized $\mu$-function as a fundamental solution of the following $q$-difference equation 
\begin{align}
\label{eq: GMmu}
[T_x^{N+1}-T_x^N+axT_x-x]f(x)=0, \quad N\in\Zz_{>0}. 
\end{align}
The case of $N=1$ is essentially equal to the $q$-Hermite-Weber equation \eqref{eq: mua equation}. 

First, we introduce $N$-th order $q$-Borel transformation $\mathcal{B}_q^N$ and $q$-Laplace transformation $\mathcal{L}^N_q$ for a formal series $g(x)=\sum A_nx^n$ as  
\begin{align*}
\mathcal{B}^N_q(g)(x)&:=\sum A_nq^\frac{Nn(n-1)}{2}x^n=\mathcal{B}_{q^N}(g)(x), \\
\mathcal{L}^N_q(g)(\lambda_0,\ldots ,\lambda_N)&:=\sum_{n\in\Zz^N}g(\lambda_0q^{|n|})\prod_{j=1}^N\frac{(\lambda_{j-1}/\lambda_j)^{n_j}q^\frac{n_j(n_j-1)}{2}}{\theta_q(\lambda_{j-1}/\lambda_j)}.
\end{align*}
It is easy to show that $\mathcal{B}_q^N$ and $\mathcal{L}_q^N$ are the $N$ time compositions of the $q$-Borel and $q$-Laplace transformations respectively: 
\begin{align}
  \mathcal{B}_q^N(g)&=\mathcal{B}_q\circ\cdots\circ\mathcal{B}_q(g),
  \\
 \mathcal{L}_q^N(g)(\lambda_0,\ldots,\lambda_N)&=\mathcal{L}_q(\lambda_N,\lambda_{N-1})\circ\mathcal{L}_q(x,\lambda_{N-2})\circ\cdots\circ\mathcal{L}_q(x,\lambda_0)(g). 
\end{align}
Hence, if $N=1$ and $\lambda_1=x$, then the transformations $\mathcal{B}^N_q$ and $\mathcal{L}^N_q$ coincide with the conventional definitions $\mathcal{B}_q$ and $\mathcal{L}_q$ respectively. 
%\eqref{q-Borel Laplace transformation}. 
Also, from a simple calculation, we have
\begin{align*}
\mathcal{B}_q^N(x^mT_x^n[g])(\xi)
  &=
  q^\frac{Nm(m-1)}{2}\xi^mT_\xi^{Nm+n}\mathcal{B}_q^N([g])(\xi),
  \\
  \mathcal{L}_q^N(\xi^mT_\xi^n[g])(\lambda_0,\ldots ,\lambda_N)
  &=
  q^{-\frac{Nm(m-1)}{2}}\lambda_N^mT_{\lambda_N}^{n-Nm}\mathcal{L}_q^N([g])(\lambda_0,\ldots ,\lambda_N), 
  \\
  \mathcal{L}_q^N\circ\mathcal{B}_q^N(x^m T_x^ng)(\lambda_0,\ldots ,\lambda_N)
  &=
  \lambda_N^m T_{\lambda_N}^n\mathcal{L}_q^N\circ\mathcal{B}_q^N(g)(\lambda_0,\ldots,\lambda_N). 
\end{align*} 
Therefore, the $q$-difference equation \eqref{multi gmu eq} has convergent solutions $\mathcal{L}_q^{N-M}\circ\mathcal{B}_q^{N-M}(\tilde{f}_k)(\lambda_0,\ldots,\lambda_{N-M})$ with $\lambda_{N-M}=x$. 

%If $M=0$ and $N=1$, then the $q$-difference equation \eqref{multi gmu eq} is essentially the $q$-Kummer difference equation (for example, see \cite{O}): 
%\begin{align*}
%[(c-axq)T_x^2-(c+q-xq)T_x+q]f(x)=0. 
%\end{align*} 
%We present fundamental solutions of \eqref{multi gmu eq} and their connection formulas. 
%we derive the multivariable generalized $\mu$-function as a fundamental solution of the following $q$-difference equation 
%\begin{align}
%\label{eq: GMmu}
%[T_x^{N+1}-T_x^N+axT_x-x]f(x)=0, \quad N\in\Zz_{>0}. 
%\end{align}
Next, we give a divergent series solution around $x=0$ of \eqref{eq: GMmu}. 
Generally,  we have the following result.
%by a bit calculation
% easy to give the following more general result. 
%which is the case of $M=0$, $b_0=1$ and $a_1=\cdots =a_N=0$ with \eqref{multi gmu eq}. 
%First, we treat the following more general $q$-difference equation: 
%for $0\leq M< N$, the $q$-difference equation
\begin{lem}
\label{lem: qK FS}
The $q$-difference equation 
\begin{align}
\label{multi gmu eq}
\left[T_x^{N-M}(1-b_0T_x)\cdots(1-b_MT_x)-x(1-a_0T_x)\cdots(1-a_NT_x)\right]f(x)=0,\quad {\text{for $0\leq M< N$}},
%=[T_x^{m+1}-q^\frac{1}{2}T_x^{m}+xq^\frac{m+2}{2}T_x-xq^\frac{m+1}{2}]f(x)=0
\end{align}
has the following divergent solutions for $k=0,\ldots ,M$: 
\begin{align*}
\frac{\theta_q(b_kx)}{\theta_q(x)}\tilde{f}_k(x):=\frac{\theta_q(b_kx)}{\theta_q(x)}{}_{N+1}\phi_M\hyper{\{a_r/b_k\}_{0\leq r\leq N}}{\{b_sq/b_k\}_{0\leq s\leq M}^k}{q}{\left(-\frac{b_k}{q}\right)^{N-M}x}, 
\end{align*}
and the following convergent solutions for $j=0,\ldots ,N$: 
\begin{align}
\label{conv solution}
\frac{\theta_q(a_jx)}{\theta_q(x)}{}_{N+1}\phi_N\hyper{\{a_j/b_s\}_{0\leq s\leq M},0\ldots,0}{\{a_jq/a_r\}_{0\leq r\leq N}^j}{q}{(-1)^{N-M}\frac{b_0\cdots b_M}{a_0\cdots a_N}\frac{q^{N+1}}{x}},
\end{align}
where $\{a_k\}_{0\leq k\leq N}^j:=\{a_0,\ldots a_{j-1},a_{j+1},\ldots,a_N\}$. 

In particular, the divergent series:
\begin{align}
\tilde{f}_N(a,x)=\sum_{n=0}^\infty\frac{(a)_n}{(q)_n}q^{-\frac{Nn(n+1)}{2}}(-x)^n
\end{align}
is a formal series solution of \eqref{eq: GMmu}. 
\end{lem}

The formal part of Lemma \ref{lem: qK FS} follows from a bit calculation. 
%The later one 
Since the $q$-difference equation \eqref{eq: GMmu} coincides with the case of $M=0$, $b_0=1$, $a_0=a$ and $a_1=\cdots =a_N=0$ in \eqref{multi gmu eq}, we have the later part of Lemma \ref{lem: qK FS} immediately. 
\begin{rmk}
If $M=0$ and $N=1$, then the $q$-difference equation \eqref{multi gmu eq} is essentially equal to the $q$-Kummer difference equation (for example, see \cite{O}): 
\begin{align*}
[(c-axq)T_x^2-(c+q-xq)T_x+q]f(x)=0. 
\end{align*} 
\end{rmk}
%We show that the $q$-difference equation \eqref{eq: GMmu} has the following divergent solution from Lemma \ref{lem: qK FS}: 
%\begin{align*}
%\tilde{f}_N(a,x)=\sum_{n=0}^\infty\frac{(a)_n}{(q)_n}q^{-\frac{Nn(n+1)}{2}}(-x)^n. 
%\end{align*}

From a divergent series solution $\tilde{f}_N(a,x)$ and the $N$-th order $q$-Borel summation method, we obtain a convergent solution $\mathcal{L}_q^N\circ\mathcal{B}_q^N(\tilde{f}_N)(\lambda_0,\ldots,\lambda_{N-1},x)$ of the $q$-difference equation \eqref{eq: GMmu}. 
The relation between $\hat{\mu}_N$ and $\mathcal{L}_q^N\circ\mathcal{B}_q^N(\tilde{f}_N)$ is as follows. 
%Then we obtain the following theorem: 
\begin{thm}
\label{thm: fN=muN}
%We have 
%\begin{align*}
%&\hat{\mu}(u_0,\ldots,u_N;\alpha;\tau)\\
%&=i^Ne^{\pi i\alpha u}q^{-\frac{N}{8}}\mathcal{L}_q^N\circ\mathcal{B}_q^N(\tilde{f}_0)(-e^{2\pi iu_0},e^{2\pi i(u_0+u_1)},-e^{2\pi i(u_0+u_1+u_2)},\ldots,(-1)^{N+1}e^{2\pi iu}). 
%\end{align*}
Let $f_N$ be 
\begin{align*}
    f_{N}&=f_{N}(x_0,\ldots,x_{N};a;q):=
    \sum_{n\in\Zz^N}\frac{(-ax_0q^{|n|})_\infty}{(-x_0q^{|n|})_\infty}\prod_{j=1}^N\frac{x_j^{-n_j}q^\frac{n_j(n_j+1)}{2}}{\theta_q(x_j)}. 
\end{align*}
%From the definition of Theorem \ref{thm: fN=muN},  
We have 
%the relations between $\mathcal{L}_q^N\circ\mathcal{B}_q^N(f_N)$, $f_N$ and $\hat{\mu}_N$ are  
\begin{align}
\label{eq: mua and fN}
\hat{\mu}_N(u_0,\ldots,u_N;\alpha;\tau)&=i^Ne^{\pi i\alpha u}q^{-\frac{N}{8}}f_N(-e^{2\pi iu_0},\ldots,-e^{2\pi iu_N};q^\alpha;q), 
\\
\label{eq: qB and fN}
\mathcal{L}_q^N\circ\mathcal{B}_q^N(\tilde{f}_N(a,\cdot))(\lambda_0,\ldots,\lambda_{N})&=f_N(\lambda_0,\lambda_1/\lambda_0,\ldots,\lambda_{N}/\lambda_{N-1};a;q), 
\end{align}
and 
\begin{align}
&\hat{\mu}(u_0,\ldots,u_N;\alpha;\tau)\nonumber\\
\label{eq: mul mua and qB}
&=i^Ne^{\pi i\alpha u}q^{-\frac{N}{8}}\mathcal{L}_q^N\circ\mathcal{B}_q^N(\tilde{f}_0)(-e^{2\pi iu_0},e^{2\pi i(u_0+u_1)},-e^{2\pi i(u_0+u_1+u_2)},\ldots,(-1)^{N+1}e^{2\pi iu}). 
\end{align}
\end{thm}
\begin{proof}[{\bf{Proof}}]
It is enough to show the equation \eqref{eq: qB and fN}. 
By using the $q$-binomial formula (for example, see \cite[p.8, (1.3.2)]{GR} and \eqref{theta relation 6}, we have
\begin{align*}
\mathcal{L}_q^N\circ\mathcal{B}_q^N(\tilde{f}_0)(\lambda_0,\ldots,\lambda_N)
  &=\mathcal{L}_q^N\left(\sum_{n=0}^\infty\frac{(a)_n}{(q)_n}\left(-\frac{x}{q^N}\right)^n\right)
  \\
  &=
  \mathcal{L}_q^N\left(\frac{(-axq^{-N})_\infty}{(-xq^{-N})_\infty}\right)
  \\
  &=
  \sum_{n\in\Zz^N}\frac{(-a\lambda_0q^{|n|-N})_\infty}{(-\lambda_0q^{|n|-N})_\infty}\prod_{j=1}^N\frac{(\lambda_{j-1}/\lambda_j)^{n_j}q^\frac{n_j(n_j-1)}{2}}{\theta_q(\lambda_{j-1}/\lambda_j)}
  \\
  &=\sum_{n\in\Zz^N}\frac{(-a\lambda_0q^{|n|})_\infty}{(-\lambda_0q^{|n|})_\infty}\prod_{j=1}^N\frac{(\lambda_{j-1}/\lambda_j)^{n_j}q^\frac{n_j(n_j+1)}{2}}{\theta_q(\lambda_{j}/\lambda_{j-1})}
  \\
  &=f_N(\lambda_0,\lambda_1/\lambda_0,\ldots,\lambda_{N}/\lambda_{N-1};a;q). 
\end{align*}
%Therefore, from the definition of $\hat{\mu}_N$ and the equation 
%\begin{align*}
%\theta_q(-e^{2\pi iu})=ie^{\pi iu}q^{-\frac{1}{8}}\vartheta(u;\tau), 
%\end{align*}
%we obtain the desired result. 
\end{proof}
If $N=1$, then the  function $f_1$ is essentially the generalized $\mu$-function from \eqref{eq: mul mua and mua} and \eqref{eq: mua and fN}: 
\begin{align}
\label{relation between the generalized mu function and the function f1}
-iq^{-\frac{1}{8}}e^{\pi i\alpha(u-v)}f_1(-e^{2\pi i(u-\alpha\tau)},-e^{2\pi i(\tau-v)};q^\alpha;q)=\mu(u,v;\alpha;\tau). 
\end{align}
%where 
%\begin{align*}
%\mu(x,y;a)=-iq^{-\frac{1}{8}}\frac{(x/y)^\frac{\alpha}{2}}{\theta_q(-y)}\sum_{n\in\Zz}\frac{(xq^n)_\infty}{(xq^n/a)_\infty}(-y)^nq^\frac{n(n-1)}{2}. 
%\end{align*}
From some formulas of the generalized $\mu$-function (Theorem 1.3 of \cite{ST}) and the relation \eqref{relation between the generalized mu function and the function f1}, we obtain the following  lemma. 
\begin{lem}
\label{lemma: f1}
We have 
%\begin{align*}
%  &(1)\ 
%    f_1(x_0q,x_1;a)=f_1(x_0,x_1q;a),
%    \\
%  &(2\text{--}1)\ 
%    [T_{x_0}T_a+ax_0x_1T_a-T_{x_0}]f_1=0,
%    \\ 
%  &(2\text{--}2)\ 
%    [(1-a+a^2x_0x_1)T_a+aT_aT_{x_0}-1]f_1=0,
%    \\
%  &(2\text{--}3)\ 
%    [(1-a+a^2x_0x_1)T_{x_0}-(1-a)T_aT_{x_0}-ax_0x_1]f_1=0,
%    \\
%  &(2\text{--}4)\ 
%    [(1-a)T_a+aT_{x_0}-1]f_1=0,
%    \\
%  &(3\text{--}1)\ 
%    [T_{x_0}^2-(1-ax_0x_1)T_{x_0}-x_0x_1]f_1=0,
%    \\
%  &(3\text{--}2)\ 
%    [(1-aq)T_a^2-(1+(1-a+a^2x)q)T_a+1]f_1=0,
%    \\
%%  &(4\text{--}1)\ 
%%    f_1(x_0,x_1;a)=\frac{(q)_\infty}{(q/a)_\infty\theta_q(-a x_0x_1/q)}\left\{{}_1\phi_1\hyper{q/a}{0}{q}{\frac{q^2}{ax_0x_1}}-\frac{\theta_q(ax_0/q)\theta_q(ax_1)}{\theta_q(x_0)\theta_q(x_1q)}{}_1\phi_1\hyper{q/a}{0}{q}{ax_0x_1}\right\},
%%    \\
%  &(4)\ 
%    f_1(x_0,x_1;a)-f_1(x_0 y^{-1},x_1 y;a)=\frac{(q)_\infty\theta_q(-a)\theta_q(-y)\theta_q(-x_1y/x_0)}{(q/a)_\infty\theta_q(x_0)\theta_q(x_1)\theta_q(y/x_0)\theta_q(x_1y)}{}_1\phi_1\hyper{q/a}{0}{q}{ax_0x_1},
%    \\
%  &(5)\ 
%    f_1(x_0,x_1;a)=f_1(x_1,x_0;a),\\
%  &(6)\
%    f_1(x_0,x_1;a)=\frac{\theta_q(ax_0)\theta_q(ax_1)}{\theta_q(x_0)\theta_q(x_1)}f_1(q/ax_0,q/ax_1;a). 
%\end{align*}
\begin{align}
 f_1(x_0,x_1;a)
 &=
 f_1(x_0 y^{-1},x_1 y;a)\nonumber\\ \label{eq: f1 1}
 &\quad+
 \frac{(q)_\infty\theta_q(-a)\theta_q(-y)\theta_q(-x_1y/x_0)}{(q/a)_\infty\theta_q(x_0)\theta_q(x_1)\theta_q(y/x_0)\theta_q(x_1y)}
 {}_1\phi_1\hyper{q/a}{0}{q}{ax_0x_1}, 
 \\
\label{eq: f1 2}
 f_1(x_0q,x_1;a)
 &=
 f_1(x_0,x_1q;a),
 \\
\label{eq: f1 3}
 f_1(x_0,x_1;a)
 &=
 f_1(x_1,x_0;a),
 \\
\label{eq: f1 4}
 f_1(x_0,x_1;a)
 &=
 \frac{\theta_q(ax_0)\theta_q(ax_1)}{\theta_q(x_0)\theta_q(x_1)}f_1(q/ax_0,q/ax_1;a). 
\end{align}
Also, $f_1$ satisfies the following $q$-difference equations: 
\begin{align}
\label{qeq: f1 1}
    [T_{x_0}T_a+ax_0x_1T_a-T_{x_0}]f_1
    &=0,
    \\ 
\label{qeq: f1 2}
    [(1-a+a^2x_0x_1)T_a+aT_aT_{x_0}-1]f_1
    &=0,
    \\
\label{qeq: f1 3}
    [(1-a+a^2x_0x_1)T_{x_0}-(1-a)T_aT_{x_0}-ax_0x_1]f_1
    &=0,
    \\
\label{qeq: f1 4}
    [(1-a)T_a+aT_{x_0}-1]f_1
    &=0,
    \\
\label{qeq: f1 5}
    [T_{x_0}^2-(1-ax_0x_1)T_{x_0}-x_0x_1]f_1
    &=0,
    \\
\label{qeq: f1 6}
    [(1-aq)T_a^2-(1+(1-a+a^2x)q)T_a+1]f_1
    &=0. 
\end{align}
\end{lem}
In general, the function $f_N$ satisfies the following formulas. 
\begin{prop}
\label{prop: fN}
We have
\begin{align}
\label{eq: fN 1}
  f_N(x_0,\ldots,x_N;a;q)
  &=\sum_{n\in\Zz^{N-1}}\left[\prod_{j=1}^{N-1}\frac{x_j^{-n_j}q^\frac{n_j(n_j+1)}{2}}{\theta_q(x_j)}\right]f_1(x_0q^{|n|},x_N;a;q), 
  \\
  f_N(x_0,\ldots,x_N;a;q)
  &=
  f_N(x_0y^{-1},x_1,\ldots,x_{N-1};x_Ny;a;q)
  \nonumber\\
     &\quad+\frac{(q)_\infty\theta_q(-a)\theta_q(-y)\theta_q(-x_Ny/x_0)}{(q/a)_\infty\theta_q(y/x_0)\theta_q(x_Ny)\theta_q(x_0)\theta_q(x_N)}\nonumber
     \\
\label{eq: fN 2}
     &\quad\times\sum_{n\in\Zz^{N-1}}
       q^\frac{{}^t\!n S n}{2}{}_1\phi_1\hyper{q/a}{0}{q}{ax_0x_Nq^{|n|}}
         \prod_{j=1}^{N-1}\frac{(-\frac{x_0x_N}{x_j})^{n_j}}{\theta_q(x_j)}, 
  \\
\label{eq: fN 3}
  T_{x_r}f_N
  &=
  T_{x_s}f_N, 
  \\
\label{eq: fN 4}
  f_N(x_0,\ldots,x_N;a;q)
  &=
  f_N(x_{\sigma(0)},\ldots,x_{\sigma(N)};a;q), 
  \\
\label{eq: fN 5}
  f_N(x_0,\ldots,x_N;a;q)
  &=
  \frac{\theta_q(ax_0)\cdots\theta_q(ax_N)}{\theta_q(x_0)\cdots\theta_q(x_N)}
  f_N(q/ax_0,\ldots,q/ax_N;a;q). 
\end{align}
Also, $f_N$ satisfies the following $q$-difference equations: 
\begin{align}
\label{eq: fN 6}
  \left[(1-a)T_a+aT_{x_0}-1\right]f_N&=0,\\
%  &(3\text{--}2)\ 
%  \left[aT_aT_{x_k}^N+(1-a)T_aT_{x_k}^{N-1}-T_{x_k}^{N-1}+a^2XT_a\right]f_N=0, \\
%  &(3\text{--}3)\ 
%  \left[(1-a)T_aT_{x_k}^N-(1-a)T_{x_k}^N+aXT_{x_k}^{N-1}-a^2XT_{x_k}\right]f_N=0,
\label{eq: fN 7}
  [T_{x_0}^{N+1}-T_{x_0}^N+aX_NT_{x_0}-X_N]f_N&=0. 
\end{align}
\end{prop}
\begin{proof}[{\bf{Proof}}]
% the definition of $f_N$, 
%First, we get \eqref{eq: fN 1} by the following calculations:
The formula \eqref{eq: fN 1} follows from the definition of $f_N$: 
%\eqref{eq: fN 1} follows from (2.3) and (2.9): 
\begin{align*}
f_N(x_0,\ldots,x_N;a;q)
  &=\sum_{n\in\Zz^{N-1}}\left[\prod_{j=1}^{N-1}\frac{x_j^{-n_j}q^\frac{n_j(n_j+1)}{2}}{\theta_q(x_j)}\right]\sum_{n_N\in\Zz}\frac{(-ax_0q^{|n|+n_N})_\infty}{(-x_0q^{|n|+n_N})_\infty}\frac{x_N^{-n_N}q^\frac{n_N(n_N+1)}{2}}{\theta_q(x_N)}
  \\
  &=\sum_{n\in\Zz^{N-1}}\left[\prod_{j=1}^{N-1}\frac{x_j^{-n_j}q^\frac{n_j(n_j+1)}{2}}{\theta_q(x_j)}\right]f_1(x_0q^{|n|},x_N;a;q). 
\end{align*}

By
% rewriting the equation \eqref{eq: fN 2} by 
\eqref{eq: fN 1}, we have
\begin{align*}
&f_N(x_0,\ldots,x_N;a;q)-f_N(x_0y^{-1},x_1,\ldots,x_{N-1},x_Ny;a;q)\\
  &=\sum_{n\in\Zz^{N-1}}\left[\prod_{j=1}^{N-1}\frac{x_j^{-n_j}q^\frac{n_j(n_j+1)}{2}}{\theta_q(x_j)}\right](f_1(x_0q^{|n|},x_N;a;q)-f_1(x_0y^{-1}q^{|n|},x_Ny;a;q)). 
\end{align*}
From the formula \eqref{eq: f1 1}, 
%and \eqref{eq: f1 4}
%, we obtain \eqref{eq: fN 2} 
%%and \eqref{eq: fN 5} 
%as the following calculations: 
\begin{align*}
  &f_1(x_0q^{|n|},x_N;a;q)-f_1(x_0y^{-1}q^{|n|},x_Ny;a;q)
  \\
  &=\frac{(q)_\infty\theta_q(-a)\theta_q(-y)\theta_q(-x_Ny/x_0q^{|n|})}{(q/a)_\infty\theta_q(x_0q^{|n|})\theta_q(x_N)\theta_q(y/x_0q^{|n|})\theta_q(x_Ny)}
  {}_1\phi_1\hyper{q/a}{0}{q}{ax_0x_Nq^{|n|}}\\
  &=\frac{(q)_\infty\theta_q(-a)\theta_q(-y)\theta_q(-x_Ny/x_0)}{(q/a)_\infty\theta_q(x_0)\theta_q(x_N)\theta_q(y/x_0)\theta_q(x_Ny)}
  (-x_0x_N)^{|n|}q^\frac{|n|(|n|-1)}{2} {}_1\phi_1\hyper{q/a}{0}{q}{ax_0x_Nq^{|n|}}. 
\end{align*}
Then we obtain the equation \eqref{eq: fN 2}. 
%\begin{align*}
%&f_N(x_0,\ldots,x_N;a;q)-f_N(x_0y^{-1},x_1,\ldots,x_{N-1},x_Ny;a;q)\\
%  &=\sum_{n\in\Zz^{N-1}}\left[\prod_{j=1}^{N-1}\frac{x_j^{-n_j}q^\frac{n_j(n_j+1)}{2}}{\theta_q(x_j)}\right](f_1(x_0q^{|n|},x_N;a;q)-f_1(x_0y^{-1}q^{|n|},x_Ny;a;q))
%  \\
%  &=\sum_{n\in\Zz^{N-1}}\left[\prod_{j=1}^{N-1}\frac{x_j^{-n_j}q^\frac{n_j(n_j+1)}{2}}{\theta_q(x_j)}\right]\\
%  &\quad\times\frac{(q)_\infty\theta_q(-a)\theta_q(-y)\theta_q(-x_Ny/x_0q^{|n|})}{(q/a)_\infty\theta_q(x_0q^{|n|})\theta_q(x_N)\theta_q(y/x_0q^{|n|})\theta_q(x_Ny)}
%  {}_1\phi_1\hyper{q/a}{0}{q}{ax_0x_Nq^{|n|}}\\
%  &=\frac{(q)_\infty\theta_q(-a)\theta_q(-y)\theta_q(-x_Ny/x_0)}{\theta_q(q/a)_\infty\theta_q(y/x_0)\theta_q(x_Ny)\theta_q(x_0)\cdots\theta_q(x_N)}\\
%     &\quad\times\sum_{n\in\Zz^{N-1}}\left[\prod_{j=1}^{N-1}x_j^{-n_j}q^\frac{n_j(n_j+1)}{2}\right](-x_0x_N)^{|n|}q^\frac{|n|(|n|-1)}{2} {}_1\phi_1\hyper{q/a}{0}{q}{ax_0x_Nq^{|n|}}\\
%  &=\frac{(q)_\infty\theta_q(-a)\theta_q(-y)\theta_q(-x_Ny/x_0)}{\theta_q(q/a)_\infty\theta_q(y/x_0)\theta_q(x_Ny)\theta_q(x_0)\cdots\theta_q(x_N)}\\
%  &\quad\times\sum_{n\in\Zz^{N-1}}
%  q^\frac{{}^t\!n S n}{2}{}_1\phi_1\hyper{q/a}{0}{q}{ax_0x_Nq^{|n|}}\prod_{j=1}^{N-1}\left(-\frac{x_0x_N}{x_j}\right)^{n_j}, 
%\end{align*}

Also, we have
\begin{align*}
&f_N(x_0,\ldots,x_N;a;q)\\
  &=\frac{\theta_q(ax_0)\cdots\theta_q(ax_N)}{\theta_q(x_0)\cdots\theta_q(x_N)}\sum_{n\in\Zz^{N-1}}\left[\prod_{j=1}^{N-1}\frac{x_j^{-n_j}q^\frac{n_j(n_j+1)}{2}}{\theta_q(ax_j)}\right]\frac{\theta_q(x_0)\theta_q(x_N)}{\theta_q(ax_0)\theta_q(ax_N)}f_1(x_0q^{|n|},x_N;a;q)
\end{align*}
from \eqref{eq: fN 1}, 
and
\begin{align*}
  \frac{\theta_q(x_0)\theta_q(x_N)}{\theta_q(ax_0)\theta_q(ax_N)}f_1(x_0q^{|n|},x_N;a;q)
  &=
  a^{-|n|}\frac{\theta_q(x_0q^{|n|})\theta_q(x_N)}{\theta_q(ax_0q^{|n|})\theta_q(ax_N)}f_1(x_0q^{|n|},x_N;a;q)
  \\
  &=
  a^{-|n|}f_1(q^{-|n|+1}/ax_0,q/ax_N;a;q)
\end{align*}
from \eqref{eq: f1 4}. 
Hence we obtain the formula \eqref{eq: fN 5}. 
%\begin{align*}
%  f_N(x_0,\ldots, x_N;a;q)
%  &=
%\end{align*}
%\begin{align*}
%  &=\frac{\theta_q(ax_0)\cdots\theta_q(ax_N)}{\theta_q(x_0)\cdots\theta_q(x_N)}\sum_{n\in\Zz^{N-1}}\left[\prod_{j=1}^{N-1}\frac{(ax_j)^{-n_j}q^\frac{n_j(n_j+1)}{2}}{\theta_q(ax_j)}\right]
%  \frac{\theta_q(x_0q^{|n|})\theta_q(x_N)}{\theta_q(ax_0q^{|n|})\theta_q(ax_N)}f_1(x_0q^{|n|},x_N;a;q)\\
%  &=\frac{\theta_q(ax_0)\cdots\theta_q(ax_N)}{\theta_q(x_0)\cdots\theta_q(x_N)}\sum_{n\in\Zz^{N-1}}\left[\prod_{j=1}^{N-1}\frac{(ax_j)^{-n_j}q^\frac{n_j(n_j+1)}{2}}{\theta_q(ax_j)}\right]f_1(q^{-|n|+1}/ax_0,q/ax_N;a;q)
%  \\
%%  &=\frac{\theta_q(ax_0)\cdots\theta_q(ax_N)}{\theta_q(x_0)\cdots\theta_q(x_N)}\sum_{n\in\Zz^{N-1}}\left[\prod_{j=1}^{N-1}\frac{(ax_j)^{n_j}q^\frac{n_j(n_j-1)}{2}}{\theta_q(ax_j)}\right]f_1(q^{|n|+1}/ax_0,q/ax_N;a;q)
%%  \\
%  &=
%  \frac{\theta_q(ax_0)\cdots\theta_q(ax_N)}{\theta_q(x_0)\cdots\theta_q(x_N)}\sum_{n\in\Zz^{N-1}}\left[\prod_{j=1}^{N-1}\frac{(q/ax_j)^{-n_j}q^\frac{n_j(n_j+1)}{2}}{\theta_q(q/ax_j)}\right]f_1(q^{|n|+1}/ax_0,q/ax_N;a;q). 
%\end{align*}

By substituting $y=q$ and $y=x_0/x_N$ into \eqref{eq: fN 2}, we get the equations \eqref{eq: fN 3} and \eqref{eq: fN 4}, respectively. 

%Therefore, if we put $y=q$ or $y=x_0/x_N$ in \eqref{eq: fN 2}, we obtain the remaining results \eqref{eq: fN 3} or \eqref{eq: fN 4} of the first half. 

The equation \eqref{eq: fN 6} follows from \eqref{qeq: f1 4} and \eqref{eq: fN 1}: 
\begin{align*}
[(1-a)T_a+aT_{x_N}-1]f_N
  &=\sum_{n\in\Zz^{N-1}}\left[\prod_{j=1}^{N-1}\frac{x_j^{-n_j}q^\frac{n_j(n_j+1)}{2}}{\theta_q(x_j)}\right][(1-a)T_a+aT_{x_N}-1]f_1(x_0q^{|n|},x_N;a;q)\\
  &=0. 
\end{align*}

Similarly, the equation \eqref{eq: fN 7} is proved by \eqref{theta relation 5}, \eqref{qeq: f1 5} and \eqref{eq: fN 1}: 
\begin{align*}
[T_{x_N}^2-T_{x_N}]f_N
  &=\sum_{n\in\Zz^{N-1}}\left[\prod_{j=1}^{N-1}\frac{x_j^{-n_j}q^\frac{n_j(n_j+1)}{2}}{\theta_q(x_j)}\right][T_{x_N}^2-T_{x_N}]f_1(x_0q^{|n|},x_N;a;q)\\
  &=\sum_{n\in\Zz^{N-1}}\left[\prod_{j=1}^{N-1}\frac{x_j^{-n_j}q^\frac{n_j(n_j+1)}{2}}{\theta_q(x_j)}\right]\left[-ax_0x_Nq^{|n|}T_{x_N}+x_0x_Nq^{|n|}\right]f_1(x_0q^{|n|},x_N;a;q)\\
  &=\left[-aX_Nq^{-(N-1)}T_{x_N}+X_Nq^{-(N-1)}\right]\sum_{n\in\Zz^{N-1}}\left[\prod_{j=1}^{N-1}\frac{(q/x_j)^{n_j}q^\frac{n_j(n_j+1)}{2}}{\theta_q(x_j/q)}\right]f_1(x_0q^{|n|},x_N;a;q)\\
  &=T_{x_1}^{-1}\cdots T_{x_{N-1}}^{-1}\left[-aX_NT_{x_N}+X_N\right]f_N.
\end{align*}
%Acting both sides by $T_{x_1}\cdots T_{x_{N-1}}$, we obtain the desired result. 
% from the equation (2) of Proposition \ref{prop: fN}. 
\end{proof}
Putting $a=q$ in Proposition \ref{prop: fN}, then we obtain the following corollary. 
\begin{coro}
\label{coro: fN}
We have 
\begin{align}
\label{eq: fNq 1}
  f_N(x_0,\ldots,x_N;q;q)
  &=
  \sum_{n\in\Zz^{N-1}}
  \left[\prod_{j=1}^{N-1}\frac{x_j^{-n_j}q^\frac{n_j(n_j+1)}{2}}{\theta_q(x_j)}\right]
  f_1(x_0q^{|n|},x_N;q;q), 
  \\
\label{eq: fNq 2}
  f_N(x_0,\ldots,x_N;q;q)
  &=
  f_N(x_0y^{-1},x_1,\ldots,x_{N-1};x_Ny;q;q)\nonumber\\
     &\quad +
     \frac{(q)_\infty^3\theta_q(-y)\theta_q(-x_Ny/x_0)}{\theta_q(y/x_0)\theta_q(x_Ny)\theta_q(x_0)\cdots\theta_q(x_N)}\sum_{n\in\Zz^{N-1}}q^\frac{{}^t\!n S n}{2}\prod_{j=1}^{N-1}\left(-\frac{x_0x_N}{x_j}\right)^{n_j}, 
  \\
\label{eq: fNq 3}
  f_N(\ldots,x_rq,\ldots;q;q)
  &=
  f_N(\ldots,x_sq,\ldots;q;q), 
  \\
\label{eq: fNq 4}
  f_N(x_0,\ldots,x_N;q;q)
  &=
  f_N(x_{\sigma(0)},\ldots,x_{\sigma(N)};q;q), 
  \\
\label{eq: fNq 5}
  f_N(x_0^{-1},\ldots,x_N^{-1};q;q)
  &=
  X_Nf_N(x_0,\ldots,x_N;q;q),
  \\
\label{eq: fNq 6}
  T_{x_0}^Nf_N(x_0,\ldots,x_N;q;q)
  &=
  -X_Nf_N(x_0,\ldots,x_N;q;q)+1.
%  \\
%  [T_{x_0}-1][T_{x_0}^N+X_N]f_N(x_0,\ldots,x_N;q;q)&=0. 
\end{align}
\end{coro}
\begin{proof}[{\bf{Proof}}]
Except for \eqref{eq: fNq 6}, we prove the formulas immediately from Proposition \ref{prop: fN}. 
From a simple calculation, 
we also give the formula \eqref{eq: fNq 6}:
\begin{align*}
T_{x_0}^Nf_N(x_0,\ldots,x_N;q;q)
  &=
  \sum_{n\in\Zz^N}
  \frac{1}{1+x_0q^{|n|}}(1+x_0q^{|n|}-x_0q^{|n|})
    \prod_{j=1}^N
    \frac{x_j^{-n_j+1}q^\frac{n_j(n_j-1)}{2}}{\theta_q(x_j)}
  \\
  &=
  \prod_{j=1}^N
    \sum_{m\in\Zz}
    \frac{x_j^{m}q^\frac{m(m-1)}{2}}{\theta_q(x_j)}
  -X_N
  \sum_{n\in\Zz^N}
  \frac{1}{1+x_0q^{|n|}}
    \prod_{j=1}^N
    \frac{x_j^{-n_j}q^\frac{n_j(n_j+1)}{2}}{\theta_q(x_j)}
 \\
 &=1-X_Nf_N(x_0,\ldots,x_N;q;q). 
\end{align*}
%(3) If $N=1$, using \eqref{theta relation 6}, we have 
%\begin{align*}
%f_1(x_0q,x_1;q;q)
%  &=\sum_{n\in\Zz}\frac{1}{1+x_0q^{n+1}}(1+x_0q^{n+1}-x_0q^{n+1})\frac{x_1^{-n}q^\frac{n(n+1)}{2}}{\theta_q(x_1)}\\
%  &=1-\sum_{n\in\Zz}\frac{x_0q^n}{1+x_0q^n}\frac{x_1^{-(n-1)}q^\frac{n(n-1)}{2}}{\theta_q(x_1)}\\
%  &=1-x_0x_1f_1(x_0,x_1;q;q). 
%\end{align*}
%In general, from (1) of Proposition \ref{prop: fN}, we have 
%\begin{align*}
%T_{x_0}^Nf_N
%  &=\sum_{n\in\Zz^{N-1}}\left[\prod_{j=1}^{N-1}\frac{x_j^{-n_j}q^\frac{n_j(n_j+1)}{2}}{\theta_q(x_j)}\right]f_1(x_0q^{|n|+N},x_N;q;q)\\
%  &=\sum_{n\in\Zz^{N-1}}\left[\prod_{j=1}^{N-1}\frac{x_j^{-(n_j-1)}q^\frac{n_j(n_j-1)}{2}}{\theta_q(x_j)}\right]f_1(x_0q^{|n|+1},x_N;q;q)\\
%  &=\sum_{n\in\Zz^{N-1}}\left[\prod_{j=1}^{N-1}\frac{x_j^{-(n_j-1)}q^\frac{n_j(n_j-1)}{2}}{\theta_q(x_j)}\right](1-x_0x_Nq^{|n|}f_1(x_0q^{|n|},x_N;q;q))\\
%  &=1-X_Nf_N. 
%\end{align*}
\end{proof}
Based on Proposition \ref{prop: fN} and Corollary \ref{coro: fN}, we prove Theorem \ref{thm: mul mua}, Corollary \ref{prop: muN} and Theorem \ref{prop: MN}.
%the multivariable generalized $\mu$-function. 
%Namely, we show the following theorem. 
\begin{proof}[{\bf{Proof of Theorem $\ref{thm: mul mua}$}}]
We rewrite Proposition \ref{prop: fN} by \eqref{eq: mua and fN} and get the conclusion. 
%Using theorem \ref{thm: fN=muN}, we rewrite Proposition \ref{prop: fN} to get Theorem $\ref{thm: mul mua}$,
\end{proof}
\begin{proof}[{\bf{Proof of Corollary $\ref{prop: muN}$}}]
Except for \eqref{mun relation 2}, we prove Corollary $\ref{prop: muN}$ immediately from Theorem $\ref{thm: mul mua}$. 
Also, since the case $\alpha=1$ of \eqref{eq: mua and fN}:
\begin{align*}
\mu_N(u_0,\ldots,u_N;\tau)=i^Ne^{\pi iu}q^{-\frac{N}{8}}f_N(-e^{2\pi iu_0},\ldots,-e^{2\pi iu_N};q;q), 
\end{align*}
the formula \eqref{mun relation 2} 
% of Corollary $\ref{prop: muN}$ 
holds from \eqref{eq: fNq 6} of Corollary $\ref{coro: fN}$. 
\end{proof}
\begin{proof}[{\bf{Proof of Theorem $\ref{prop: MN}$}}]
It is clear from the definition of the function $M_N$ and Corollary \ref{prop: muN}.
\end{proof}
%%%%%%%%%%%%%%%%%%%%%%%%%%%%%%%%%%%%%%% New Section %%%%%%%%%%%%%%%%%%%%%%%%%%%%%%%%%%%%%%%%%%%%%%%%%
\section
%Vector valued multivariable mock Jacobi forms}
{Proofs of Theorem \ref{thm: MN modular} and Theorem \ref{thm: modular completion}}
\label{sec: 3}
In this section, we prove modular transformations about the function $M_N$. 
First, we present two lemmas and one proposition. %to prove Theorem \ref{thm: MN modular}. 
\begin{lem}[\cite{Zw}, \cite{Mo}]
\label{lem: h relation}
The Mordell integral $h(u;\tau)$ \eqref{defi: h} satisfies the following transformations: 
\begin{align}
\label{eq: h relation 1}
h(u;\tau)+h(u+1;\tau)
  &=\frac{2}{\sqrt{-i\tau}}e^{\frac{\pi i}{\tau}(u+\frac{1}{2})^2}, \\
\label{eq: h relation 2}
h(u;\tau)+e^{-2\pi iu}q^{-\frac{1}{2}}h(u+\tau;\tau)
  &=2e^{-\pi iu}q^{-\frac{1}{8}}. 
\end{align}
Moreover, $h(u;\tau)$ is the unique holomorphic function which satisfies \eqref{eq: h relation 1} and \eqref{eq: h relation 2}. 
\end{lem}
\begin{lem}
\label{lem: nu}
%{\color{red} $\nu_{N,k}$に関する性質}
The function $\nu_{N,k}$ \eqref{defi: nu} satisfies the following formulas: 
\begin{align}
\label{lem: nu 1}
\nu_{N,k}(u_0+1,u_1,\ldots,u_N;\tau)
  &=\zeta_N^{-k}\nu_{N,k}(u_0,\ldots,u_N;\tau), 
  \\
\label{lem: nu 2}
  \nu_{N,k}(u_0+\tau,u_1,\ldots,u_N;\tau)
  &=e^{-2\pi i(v,e_1)}q^{-\frac{N-1}{2N}}\nu_{N,k+1}(u_0,\ldots,u_N;\tau), 
  \\
\label{lem: nu 3}
  \nu_{N,k}(u_0,\ldots,u_N;\tau+1)
  &=(-1)^k\zeta_{2N}^{-k^2}\nu_{N,k}(u_0,\ldots,u_N;\tau),
  \\
\label{lem: nu4}
  \nu_{N,k}\left(\frac{u_0}{\tau},\ldots,\frac{u_N}{\tau};-\frac{1}{\tau}\right)
  &=\frac{(-i\tau)^\frac{N-1}{2}}{N^\frac{1}{2}}e^{\frac{\pi i}{\tau}{}^t\! v \widehat{S}^{-1}v}\sum_{j=0}^{N-1}\zeta_N^{jk}\nu_{N,j}(u_0,\ldots,u_N;\tau).
\end{align}
\end{lem}
\begin{proof}[{\bf{Proof}}]
For \eqref{lem: nu 1}--\eqref{lem: nu 3}, it is clear from the definition of $\nu_{N,k}$. 
Also, from a simple calculation, we have 
\begin{align*}
\widehat{S}^{-1}&=\begin
{bmatrix}
\frac{N-1}{N} & -\frac{1}{N} & \ldots & -\frac{1}{N}\\
-\frac{1}{N} & \frac{N-1}{N} & \ldots & -\frac{1}{N} \\
 & & \ldots & \\
-\frac{1}{N} & -\frac{1}{N} & \ldots & \frac{N-1}{N}
\end{bmatrix}
,\quad \det (\widehat{S})=N,
\end{align*}
and
\begin{align}
\label{nu to thetaS}
&\nu_{N,k}(u_0,\ldots,u_N;\tau)\nonumber\\
 &=e^{2\pi ik(u_0+u_1)-\frac{2\pi iku}{N}}q^\frac{k^2(N-1)}{2N}\theta_{\widehat{S}}(u_0+u_1-u_2+k\tau\ldots,u_0+u_1-u_N+k\tau;\tau). 
\end{align}
For $n\in\Zz^{N-1}$ and $e_k={}^t\![k/N,\ldots,k/N]$, the image $\widehat{S}^{-1} n$ is an element of the set $\bigsqcup_{k=0}^{N-1}(\Zz^{N-1}+e_k)$. 
Moreover, for $n\in\bigsqcup_{k=0}^{N-1}(\Zz^{N-1}+e_k)$, the image $\widehat{S}n$ is an element of the set $\Zz^{N-1}$. 
Namely, the function $\widetilde{\theta}_{\widehat{S}}$ is expressed as 
\begin{align}
\widetilde{\theta}_{\widehat{S}}(v_1,\ldots,v_{N-1};\tau)
  &=\sum_{k=0}^{N-1}\sum_{n\in\Zz^{N-1}+e_k}\exp(\pi i{}^t\! n\widehat{S} n\tau+2\pi i(v,n))
   \nonumber\\
\label{thetaS trans}
  &=\sum_{k=0}^{N-1}e^{2\pi i(v,e_k)}q^\frac{k^2(N-1)}{2N}\theta_{\widehat{S}}(v_1+k\tau,\ldots,v_{N-1}+k\tau;\tau). 
\end{align}
Hence, using \eqref{nu to thetaS}, \eqref{thetaS trans} and \eqref{mul theta relation 4}, we have
\begin{align*}
\nu_{N,k}\left(\frac{u_0}{\tau},\ldots,\frac{u_N}{\tau};-\frac{1}{\tau}\right)
  &=e^{\frac{2\pi ik(u_0+u_1)}{\tau}-\frac{2\pi iku}{N\tau}-\frac{\pi ik^2(N-1)}{N\tau}}\theta_{\widehat{S}}\left(\frac{v_1}{\tau}-\frac{k}{\tau},\ldots,\frac{v_{N-1}}{\tau}-\frac{k}{\tau};-\frac{1}{\tau}\right)
  \\
  &=\frac{(-i\tau)^\frac{N-1}{2}}{N^\frac{1}{2}}e^{\frac{\pi i}{\tau}{}^t\! v \widehat{S}^{-1}v}\widetilde{\theta}_{\widehat{S}}(v_1-k,\ldots,v_{N-1}-k;\tau)
  \\
  &=\frac{(-i\tau)^\frac{N-1}{2}}{N^\frac{1}{2}}e^{\frac{\pi i}{\tau}{}^t\! v \widehat{S}^{-1}v}\sum_{j=0}^{N-1}\zeta_N^{jk}e^{2\pi i(v,e_j)}q^\frac{j^2(N-1)}{2N}\theta_{\widehat{S}}(v_1+j\tau,\ldots,v_{N-1}+j\tau;\tau)
  \\
  &=\frac{(-i\tau)^\frac{N-1}{2}}{N^\frac{1}{2}}e^{\frac{\pi i}{\tau}{}^t\! v \widehat{S}^{-1}v}\sum_{j=0}^{N-1}\zeta_N^{jk}\nu_{N,j}(u_0,\ldots,u_N;\tau).
\end{align*}
\end{proof}
From Lemma \ref{lem: nu}, we obtain the following Proposition satisfied by $\Phi_N$ \eqref{defi: Phi}. 

\begin{prop}
\label{lem: PhiN}
The function $\Phi_N$ \eqref{defi: Phi} satisfies the following formulas: 
\begin{align}
\label{eq: PhiN 1}
  \Phi_N(u_0+1,u_1,\ldots,u_N;z;\tau)
  &=
  -\left[\delta_{jk}\zeta_N^{-k}\right]_{j,k=0}^{N-1}\Phi_N(u_0,\ldots,u_N;z;\tau), 
  \\
\label{eq: PhiN 2}
  \Phi_N(u_0+\tau,u_1,\ldots,u_N;z;\tau)
  &=
  -e^\frac{2\pi iu}{N}q^\frac{1}{2N}
\begin{bmatrix}
 0 & 1 & 0 & \ldots & 0\\
 0 & 0 & 1 & \ldots & 0\\
& &  \ldots  \\
 0 & 0 & 0 & \ldots & 1\\
 1 & 0 & 0 & \ldots & 0
 \end{bmatrix}
  \Phi_N(u_0,\ldots ,u_N;z;\tau), 
  \\
\label{eq: PhiN 3}
  \Phi_N(u_0,\ldots ,u_N;z;\tau+1)
  &=
  e^{-\frac{\pi iN}{4}}
\left[\delta_{jk}(-1)^k\zeta_{2n}^{-k^2}\right]_{j,k=0}^{N-1}
 \Phi_N(u_0,\ldots ,u_N;z;\tau), 
  \\
\label{eq: PhiN 4}
  \Phi_N\left(\frac{u_0}{\tau},\ldots ,\frac{u_N}{\tau};\frac{z}{\tau};-\frac{1}{\tau}\right)
  &=
  \frac{\sqrt{-i\tau}}{(-i)^{N+1}N^\frac{1}{2}}e^{-\frac{\pi iu^2}{N\tau}}[\zeta_N^{jk}]_{j,k=0}^{N-1}\Phi_N(u_0,\ldots ,u_N;z;\tau). 
\end{align}
\end{prop}

From Lemma \ref{lem: h relation} and Proposition \ref{lem: PhiN}, we prove modular transformations satisfied by $M_N$.
\begin{proof}[{\bf{Proof of Theorem $\ref{thm: MN modular}$}}]
\eqref{eq: MN modular 1}: It is obvious from a simple following relation:
\begin{align*}
\mu_N(u_0+k(\tau+1),u_1,\ldots,u_N;\tau+1)=(-1)^ke^{-\frac{\pi iN}{4}}\mu_N(u_0+k\tau,u_1,\ldots,u_N;\tau). 
\end{align*}
\eqref{eq: MN modular 2}: From \eqref{eq: PhiN 4} and \eqref{eq: MN 3}, the function 
\begin{align}
\label{defi: HNtilde}
\frac{(-i)^{N+1}e^\frac{\pi iu^2}{N\tau}}{N^\frac{1}{2}\sqrt{-i\tau}}\left[\zeta_N^{-jk}\right]_{j,k=0}^{N-1}M_N\left(\frac{u_0}{\tau},\ldots,\frac{u_N}{\tau};-\frac{1}{\tau}\right)-M_N(u_0,\ldots,u_N;\tau)
\end{align}
is a one variable function with respect to $u=u_0+\cdots+u_N$. 
We define this function \eqref{defi: HNtilde} as $\widetilde{H}_N(u;\tau)$. 
From \eqref{mun relation 1} and \eqref{mun relation 2}, 
%of Proposition \ref{prop: muN}
the function $\widetilde{H}_N$ satisfies the following relations: 
\begin{align}
\label{eq: tH relation 1}
\widetilde{H}_N(u+N\tau)-(-1)^Ne^{2\pi iu}q^\frac{N}{2}\widetilde{H}_N(u;\tau)
  &=-i^Ne^{\pi iu}q^\frac{3N}{8}\left[(-1)^ke^{-\frac{2\pi iku}{N}}q^{-\frac{k(k+N)}{2N}}\right]_{k=0}^{N-1}, \\
\label{eq: tH relation 2}
\widetilde{H}_N(u+N;\tau)-(-1)^N\widetilde{H}_N(u;\tau)
  &=\frac{(-1)^Ni}{N^\frac{1}{2}\sqrt{-i\tau}}\left[\zeta_N^{-jk}\right]_{j,k=0}^{N-1}\left[(-1)^ke^{\frac{\pi i}{N\tau}(u-k+\frac{N}{2})^2}\right]_{k=0}^{N-1}. 
\end{align}

On the other hand, the right hand side of the equation \eqref{eq: MN modular 2} satisfies the above relations \eqref{eq: tH relation 1} and \eqref{eq: tH relation 2}. 
In fact, putting
\begin{align*}
H_N(u;\tau):=\left[(-1)^ke^{-\frac{2\pi iku}{N}}q^{-\frac{k^2}{2N}}h\left(u+k\tau-\frac{N-1}{2};N\tau\right)\right]_{k=0}^{N-1}, 
\end{align*}
the function $H_N$ satisfies the following relations from \eqref{eq: h relation 1} and \eqref{eq: h relation 2}: 
\begin{align}
\label{eq: HN relation 1}
H_N(u+N\tau;\tau)-(-1)^Ne^{2\pi iu}q^\frac{N}{2}H_N(u;\tau)&=-2(-i)^Ne^{\pi iu}q^\frac{3N}{8}\left[(-1)^ke^{-\frac{2\pi iku}{N}}q^{-\frac{k}{2}-\frac{k^2}{2N}}\right]_{k=0}^{N-1},\\
\label{eq: HN relation 2}
H_N(u+N;\tau)-(-1)^NH_N(u;\tau)&=\frac{2}{N^\frac{1}{2}\sqrt{-i\tau}}\left[\sum_{j=0}^{N-1}(-1)^j\zeta_N^{-jk}e^{\frac{\pi i}{N\tau}(u-j+\frac{N}{2})^2}\right]_{k=0}^{N-1}.
\end{align}
Moreover, $H_N$ is the unique holomorphic function which satisfies the equations \eqref{eq: HN relation 1} and \eqref{eq: HN relation 2}. 
Therefore, we obtain
\begin{align*}
\widetilde{H}_N(u;\tau)=\frac{(-1)^Ni}{2}H_N(u;\tau).
\end{align*}
\end{proof}
\begin{rmk}
Finding the appropriate Mordell integral $H_N(u;\tau)$ from the difference equations \eqref{eq: tH relation 1} and \eqref{eq: tH relation 2} may be seem difficult from this proof, 
but we can find it by the following heuristic argument. 
%but it is actually quite natural. 
Since the right hand side of \eqref{eq: tH relation 2} is rewritten as 
\begin{align*}
(-1)^Ni\left[\zeta_N^{-jk}\right]_{j,k=0}^{N-1}\left[\int_{\Rz}(-1)^k e^{\pi iN\tau x^2-2\pi x(u-k+\frac{N}{2})}dx\right]_{k=0}^{N-1}, 
\end{align*}
the function $\widetilde{H}_N$ satisfies the following relation for $m\in\Zz_{\geq0}$: 
\begin{align*}
&\widetilde{H}_N(u;\tau)-(-1)^{mN}\widetilde{H}_N(u+mN;\tau)\\
  &\qquad=-\frac{e^{-\pi iu}q^{-\frac{N}{8}}}{2i^{N-1}}\left[\zeta_N^{-jk}\right]_{j,k=0}^{N-1}\left[\int_{\Rz-\frac{i}{2}}\frac{e^{\pi iN\tau x^2-2\pi x(u-k+\frac{N\tau}{2})}}{\sinh(\pi Nx)}\left(1-e^{-2\pi mNx}\right)\right]_{k=0}^{N-1}.
\end{align*}
Based on the integral of the right hand side, we derive the following vector valued function whose components are Mordell integrals for a sufficiently small number $\varepsilon>0$: 
\begin{align}
\label{func: vector valued Mordell integral}
\left[\zeta_N^{-jk}\right]_{j,k=0}^{N-1}\left[\int_{\Rz-\varepsilon i}\frac{e^{\pi iNx^2\tau-2\pi x(u-k+\frac{N\tau}{2})}}{\sinh(\pi Nx)}dx\right]_{k=0}^{N-1}.
\end{align}
Also, since 
\begin{align*}
h(u;N\tau)&=-2ie^{-\pi iu}q^{-\frac{N}{8}}\int_{\Rz-\varepsilon i}\frac{e^{\pi iNx^2\tau-2\pi x(u+\frac{N\tau}{2})}}{e^{\pi x}-e^{-\pi x}}dx, \\
\frac{1}{a-a^{-1}}&=\frac{1}{a^N-a^{-N}}\sum_{j=0}^{N-1}a^{2j-N+1},
\end{align*}
%from simple calculations, 
we have
\begin{align*}
h\left(u+k\tau-\frac{N-1}{2};N\tau\right)=(-1)^{k+1}i^Ne^{-\pi iu+\frac{2\pi iku}{N}}q^{\frac{k^2}{2N}-\frac{N}{8}}\sum_{j=0}^{N-1}\zeta_{N}^{-jk}\int_{\Rz-\varepsilon i}\frac{e^{\pi iNx^2\tau-2\pi x(u-j+\frac{N\tau}{2})}}{\sinh(\pi N x)}dx. 
\end{align*}
Therefore, we get
\begin{align*}
-e^{-\pi iu}q^{-\frac{N}{8}}\left[\zeta_N^{-jk}\right]_{j,k=0}^{N-1}\left[\int_{\Rz-\varepsilon i}\frac{e^{\pi iNx^2\tau-2\pi x(u-k+\frac{N\tau}{2})}}{\sinh(\pi Nx)}dx\right]_{k=0}^{N-1}=(-i)^NH_N(u;\tau), 
\end{align*}
 and obtain the appropriate Mordell integral. 
%{\color{red}{差分方程式\eqref{eq: tH relation 1}と\eqref{eq: tH relation 2}から適切なMordell積分を求める方法を書く. }}
\end{rmk}

Next, we give two lemmas to prove Theorem \ref{thm: modular completion}. 
\begin{lem}[\cite{Zw}]
\label{lem: R transformation}
The function $R(u;\tau)$ \eqref{defi: R} has the following formulas: 
\begin{align}
\label{eq: R 1}
  R(u+1;\tau)
  &=
  -R(u), 
  \\
\label{eq: R 2}
R(u+\tau;\tau)
&=
-e^{2\pi iu}q^\frac{1}{2}R(u;\tau)+2e^{\pi iu}q^\frac{3}{8}, 
\\
\label{eq: R 3}
  R(u;\tau+1)
  &=
  e^{-\frac{\pi i}{4}}R(u;\tau), 
  \\
\label{eq: R 4}
  \frac{e^{\frac{\pi iu^2}{\tau}}}{\sqrt{-i\tau}}R\left(\frac{u}{\tau};-\frac{1}{\tau}\right)
  &=-R(u;\tau)+h(u;\tau). 
\end{align}
%where the definition of the function $R$ is \eqref{defi: R}. 
\end{lem}
\begin{lem}
\label{lem: R relation}
We have 
\begin{align*}
&(-i)^{N+1}e^{-\pi i(N-1)u}q^{\frac{(N-1)^2}{8N}}R\left(u;\frac{\tau}{N}\right)\\
&=\sum_{k=0}^{N-1}(-1)^ke^{-2\pi iku}q^{-\frac{k(k-N+1)}{2N}}R\left(Nu+k\tau-\frac{N-1}{2}\tau+\frac{N+1}{2};N\tau\right). 
\end{align*}
\end{lem}
\begin{proof}[{\bf{Proof}}]
By the straight forward calculation, we have 
\begin{align*}
R(u;\tau)
%&=\sum_{n\in\Zz}\left\{\sgn\left(n+\frac{1}{2}\right)-E\left(\left(n+\frac{1}{2}+a\right)\sqrt{2y}\right)\right\}(-1)^ne^{-\pi i\left(n+\frac{1}{2}\right)^2\tau-2\pi i\left(n+\frac{1}{2}\right)u}
%\\
    &=\sum_{k=0}^{N-1}\sum_{n\in\Zz}\left\{\sgn\left(Nn+k+\frac{1}{2}\right)-E\left(\left(Nn+k+\frac{1}{2}+a\right)\sqrt{2y}\right)\right\}
    \\
    &\quad \times(-1)^{Nn+k}e^{-2\pi i\left(Nn+k+\frac{1}{2}\right)u}q^{-\frac{1}{2}\left(Nn+k+\frac{1}{2}\right)^2}
    \\
    &=i^{N+1}e^{\pi i(N-1)u}q^{-\frac{(N-1)^2}{8}}\sum_{k=0}^{N-1}(-1)^ke^{-2\pi iku}q^{-\frac{k(k-N+1)}{2}}\sum_{n\in\Zz}
    \left\{\sgn\left(Nn+k+\frac{1}{2}\right)\right.\\
    &\quad\left.-E\left(\left(Nn+k+\frac{1}{2}+a\right)\sqrt{2y}\right)\right\}
    (-1)^ne^{-2\pi i(n+\frac{1}{2})(Nu+N(k-\frac{N-1}{2})\tau+\frac{N+1}{2})}q^{-\frac{N^2}{2}(n+\frac{1}{2})^2}
    \\
    &=i^{N+1}e^{\pi i(N-1)u}q^{-\frac{(N-1)^2}{8}}\sum_{k=0}^{N-1}(-1)^ke^{-2\pi iku}q^{-\frac{k(k-N+1)}{2}}\\
    &\quad\times R\left(Nu+\frac{N(2k-N+1)}{2}\tau+\frac{N+1}{2};N^2\tau\right).
\end{align*}
Therefore, by putting $\tau\mapsto\tau/N$, we obtain the conclusion. 
\end{proof}
By using the Lemma \ref{lem: R transformation} and Lemma \ref{lem: R relation}, we prove the modular completion of $M_N$ \eqref{modular completion of MN}. 
\begin{proof}[{\bf{Proof of Theorem $\ref{thm: modular completion}$}}]
\eqref{eq: MN completion 1}--\eqref{eq: MN completion 3}: These formulas follow from \eqref{eq: MN 1}, \eqref{eq: MN 2}, \eqref{eq: MN modular 1} and \eqref{eq: R 1}--\eqref{eq: R 3}. \\
\eqref{eq: MN completion 4}: 
It is enough to show that 
\begin{align}
&\left[(-1)^{k}e^{-\frac{2\pi iku}{N}}q^{-\frac{k^2}{2N}}\left(R\left(u+k\tau+\frac{N+1}{2};N\tau\right)-(-1)^Nh\left(u+k\tau-\frac{N-1}{2};N\tau\right)\right)\right]_{k=0}^{N-1}\nonumber\\
\label{H and R}
  &=\frac{(-i)^{N+1}e^\frac{\pi iu^2}{N\tau}}{N^\frac{1}{2}\sqrt{-i\tau}}\left[\zeta_{N}^{-jk}\right]_{j,k=0}^{N-1}\left[(-1)^ke^{-\frac{2\pi iku}{N\tau}+\frac{\pi ik^2}{N\tau}}R\left(\frac{u}{\tau}-\frac{k}{\tau}+\frac{N+1}{2};-\frac{N}{\tau}\right)\right]_{k=0}^{N-1}.
\end{align}
The formula \eqref{H and R} is derived from the following two equations: 
\begin{align*}
&
%(-1)^{N+k}\frac{i}{2}e^{-\frac{2\pi iku}{N}-\frac{\pi ik^2\tau}{N}}
h\left(u+k\tau-\frac{N-1}{2};N\tau\right)\\
    %\\&
    &=
    %(-1)^{N+k}\frac{i}{2}e^{-\frac{2\pi iku}{N}-\frac{\pi ik^2\tau}{N}}\left\{
    R\left(u+k\tau-\frac{N-1}{2};N\tau\right)
    +
    \frac{e^{\frac{\pi i}{N\tau}(u+k\tau-\frac{N-1}{2})^2}}{N^\frac{1}{2}\sqrt{-i\tau}}R\left(\frac{u+k\tau-\frac{N-1}{2}}{N\tau};N\tau\right), 
    \\
      &e^{\frac{\pi i}{N\tau}(u+k\tau-\frac{N-1}{2})^2}R\left(\frac{u+k\tau-\frac{N-1}{2}}{N\tau};N\tau\right)
  \\
  &=(-1)^ki^{N+1}e^{\frac{\pi iu^2}{N\tau}+\frac{2\pi iku}{N}}q^\frac{k^2}{2N}\sum_{j=0}^{N-1}(-1)^j\zeta_N^{-jk}e^{-\frac{2\pi iju}{N\tau}+\frac{\pi ij^2}{N\tau}}R\left(\frac{u}{\tau}-\frac{j}{\tau}+\frac{N+1}{2};-\frac{N}{\tau}\right). 
    %\right\}
\end{align*}
The first one follows from the equation \eqref{eq: R 4} of Lemma \ref{lem: R transformation}, and the second one follows from Lemma \ref{lem: R relation}. 
%Therefore, we obtain the conclusion.
\end{proof}

\section{Modular transformations of the multivariable $\mu$-function}
\label{sec: 4}
In this section, we give some modular transformations of the multivariable $\mu$-function as applications of Theorem \ref{thm: MN modular} and Theorem \ref{thm: modular completion}. 
\subsection{The case of odd $N$}
In this subsection, we assume that $N$ is a positive odd number.  
\begin{coro} 
\label{coro: odd}
We define
\begin{align}
\label{eq: modular completion of muN}
\widetilde{\mu}_N(u_0,\ldots,u_N;\tau):=\mu_N(u_0,\ldots,u_N;\tau)+\frac{i}{2}R\left(u+\frac{N+1}{2};N\tau\right). 
\end{align}
Then, for $m\in\Zz$, the function $\widetilde{\mu}_N$ satisfies the following relations: 
\begin{align}
\label{eq: odd trans 1}
  e^\frac{\pi imN}{4}\widetilde{\mu}_N(u_0,\ldots,u_N;\tau+m)
  &=
  \widetilde{\mu}_N(u_0,\ldots,u_N;\tau), 
  \\
\label{eq: odd trans 2}
  \frac{i^\frac{m}{2}e^\frac{\pi imu^2}{mN\tau-1}}{\sqrt{-mN\tau+1}}
  \widetilde{\mu}_N\left(\frac{u_0}{mN\tau-1},\ldots,\frac{u_N}{mN\tau-1};-\frac{\tau}{mN\tau-1}\right)
  &=
  \widetilde{\mu}_N(u_0,\ldots,u_N;\tau). 
\end{align}
\end{coro}
\begin{proof}[{\bf{Proof}}]
\eqref{eq: odd trans 1}: It is clear from \eqref{eq: MN completion 3}. \\
\eqref{eq: odd trans 2}: We put $W_N:=\left[\delta_{jk}(-1)^k\zeta_{2N}^{k^2}\right]_{j,k=0}^{N-1}$ and $Z_N:=\left[\zeta_N^{-jk}\right]_{j,k=0}^{N-1}$. 
Regardless of the parity of $N$, we have
\begin{align}
&\widetilde{M}_N(u_0,\ldots,u_N;\tau)\nonumber\\
\label{eq: modular STS}
  &=
  \frac{(-1)^{N+1}e^\frac{\pi imN}{4}}{N\sqrt{-m\tau+1}}e^\frac{\pi imu^2}{N(m\tau-1)}Z_NW_N^mZ_N\widetilde{M}_N\left(\frac{u_0}{m\tau-1},\ldots,\frac{u_N}{m\tau-1};-\frac{\tau}{m\tau-1}\right)
  \\
\label{eq: modular STS 2}
  &=e^\frac{\pi imN}{4}\frac{e^\frac{\pi imu^2}{N(m\tau-1)}}{N\sqrt{-m\tau+1}}\left[\sum_{l=0}^{N-1}(-1)^{ml}\zeta_{2N}^{ml^2-2l(j+k)}\right]_{j,k=0}^{N-1}\widetilde{M}_N\left(\frac{u_0}{m\tau-1},\ldots,\frac{u_N}{m\tau-1};-\frac{\tau}{m\tau-1}\right)
\end{align}
from \eqref{eq: MN completion 3} and \eqref{eq: MN completion 4}.
Then, rewriting $m\mapsto mN$, the $(0,k)$ entry of the matrix of \eqref{eq: modular STS 2} is 
\begin{align*}
\sum_{l=0}^{N-1}(-1)^{mNl}\zeta_{2N}^{mNl^2-2lk}
  &=\sum_{l=0}^{N-1}\zeta_N^{-lk}
  =\begin{cases}
N & k=0\\
0 & k=1,\ldots,N-1
\end{cases}.
\end{align*}
Therefore, we obtain the conclusion. 
\end{proof}
\begin{coro}
\label{coro: odd 2}
For $m\in\Zz_{>0}$, the multivariable $\mu$-function satisfies the following relations: 
\begin{align}
\label{eq: odd mun 1}
  \mu_N(u_0,\ldots,u_N;\tau)
  &=
  e^\frac{\pi imN}{4}\mu_N(u_0,\ldots,u_N;\tau+m),
  \\
  \mu_N(u_0,\ldots,u_N;\tau)
  &=
  \frac{i^\frac{m}{2}e^\frac{\pi imu^2}{mN\tau-1}}{\sqrt{-mN\tau+1}}\mu_N\left(\frac{u_0}{mN\tau-1},\ldots,\frac{u_N}{mN\tau-1};-\frac{\tau}{mN\tau-1}\right)
  \nonumber\\
\label{eq: odd mun 2}
  &\quad-\frac{i^{\frac{m+1}{2}-N}}{2\sqrt{m}}\sum_{j=1}^{m}(-1)^je^\frac{\pi i(2j-1)^2}{4m}h\left(u+\frac{2j-1}{2m}-\frac{1}{2};N\tau-\frac{1}{m}\right).
\end{align}
\end{coro}
\begin{proof}[{\bf{Proof}}]
\eqref{eq: odd mun 1}: It is clear from \eqref{eq: MN modular 1}. \\
\eqref{eq: odd mun 2}: For $m>0$, we have
\begin{align}
\label{eq: +mNtau}
\mu_N(u_0+mN\tau,u_1,\ldots,u_N;\tau)
  &=(-1)^{mN}e^{2\pi imu}q^\frac{m^2N}{2}\mu_N(u_0,\ldots,u_N;\tau)
\nonumber\\
  &\quad+\sum_{j=1}^m e^{\pi i(2j-1)(u+\frac{N}{2})}q^{\frac{N}{8}(2j-1)(4m-2j+1)}
\end{align}
and 
\begin{align}
\label{eq: -mNtau}
\mu_N(u_0-mN\tau,u_1,\ldots,u_N;\tau)
  &=(-1)^{mN}e^{-2\pi imu}q^\frac{m^2N}{2}\mu_N(u_0,\ldots,u_N;\tau)
  \nonumber\\
  &\quad-\sum_{j=1}^me^{-\pi i(2j-1)(u+\frac{N}{2})}q^{\frac{N}{8}(2j-1)(4m-2j+1)}
\end{align}
using \eqref{mun relation 2} regardless of the parity of $N$. 
From \eqref{eq: odd trans 2}, the function
\begin{align}
\label{defi: hNtilde}
  \frac{i^\frac{m}{2}e^\frac{\pi imu^2}{mN\tau-1}}{\sqrt{-mN\tau+1}}\mu_N\left(\frac{u_0}{mN\tau-1},\ldots,\frac{u_N}{mN\tau-1};-\frac{\tau}{mN\tau-1}\right)
  -
  \mu_N(u_0,\ldots,u_N;\tau)
\end{align} 
is a one variable function with respect to $u$. 
We define this function \eqref{defi: hNtilde} as $\widetilde{h}_N(u;\tau)$ that is a holomorphic function in $u\in\Cz$. 
From \eqref{mun relation 1} and \eqref{mun relation 2} of Corollary \ref{prop: muN} and the equations \eqref{eq: +mNtau} and \eqref{eq: -mNtau}, the function $\widetilde{h}_N$ satisfies the following relations: 
\begin{align*}
\widetilde{h}_N(u+mN\tau-1;\tau)
   &=-(-1)^me^{2\pi imu}q^\frac{m^2N}{2}\widetilde{h}_N(u;\tau)
   +\sum_{j=1}^m e^{\pi i(2j-1)(u+\frac{N}{2})}q^{\frac{N}{8}(2j-1)(4m-2j+1)},
   \\
\widetilde{h}_N(u)+\widetilde{h}_N(u+1)
   &=\frac{i^{\frac{m}{2}-N}}{\sqrt{-mN \tau+1}}\sum_{j=1}^m(-1)^je^{\frac{\pi i(2j-1)^2}{4m}+\frac{\pi im}{mN\tau-1}(u+\frac{2j-1}{2m})^2}.
\end{align*}
On the other hand, we define
\begin{align*}
h_N(u;\tau):=\sum_{j=1}^{m}(-1)^je^\frac{\pi i(2j-1)^2}{4m}h\left(u+\frac{2j-1}{2m}-\frac{1}{2};N\tau-\frac{1}{m}\right).
\end{align*}
From Lemma \ref{lem: h relation}, the function $h_N$ satisfies the following relations: 
\begin{align*}
h_N(u+1;\tau)+h_N(u;\tau)&=\frac{i^{-\frac{1}{2}}2\sqrt{m}}{\sqrt{-m N\tau+1}}\sum_{j=1}^m(-1)^je^{\frac{\pi i(2j-1)^2}{4m}+\frac{\pi im}{m N\tau-1}(u+\frac{2j-1}{2m})^2},\\
h_N(u+mN\tau-1;\tau)&=-(-1)^me^{2\pi imu}q^\frac{m^2N}{2}h_N(u;\tau)\\
    &\quad+2\sqrt{m}i^Ne^{-\frac{\pi i}{4}(m+1)}\sum_{j=1}^m e^{\pi i(2j-1)(u+\frac{N}{2})}q^{\frac{N}{8}(2j-1)(4m-2j+1)}.
\end{align*}
Then, we obtain the conclusion. 
\end{proof}

%\begin{coro}
%We have 
%\begin{align*}
%aaa
%\end{align*}
%\end{coro}
%\begin{proof}
%From (1) and (2) of Theorem \ref{thm: MN modular}, we have 
%\begin{align*}
%&\frac{e^{\frac{\pi imu^2}{N(m\tau-1)}+\frac{\pi imN}{4}}}{N\sqrt{-m\tau+1}}Z_NW_N^mZ_NM_N\left(\frac{u_0}{m\tau-1},\ldots,\frac{u_N}{m\tau-1};-\frac{\tau}{m\tau-1}\right)+(-1)^N M_N\\
%&=-\frac{i}{2}\left\{H_N(u;\tau)+\frac{(-i)^{N+1}e^{\frac{\pi iu^2}{N\tau}+\frac{\pi imN}{4}}}{N^\frac{1}{2}\sqrt{-i\tau}}Z_NW_N^m H_N\left(\frac{u}{\tau};m-\frac{1}{\tau}\right)\right\}
%\end{align*}
%regardless of the parity of $N$. 
%\end{proof}
\subsection{The case of even $N$}
In this subsection, we assume that $N$ is a positive even number. 
By similar arguments, we obtain the following formulas immediately.  
\begin{coro}
\label{coro: even}
For $m\in\Zz$, the function $\widetilde{\mu}_N$ satisfies the following relations: 
\begin{align}
\label{eq: even trans 1}
  \widetilde{\mu}_N(u_0,\ldots,u_N;\tau)
  &=e^\frac{\pi imN}{4}\widetilde{\mu}_N(u_0,\ldots,u_N;\tau+m), 
  \\
\label{eq: even trans 2}
  \widetilde{\mu}_N(u_0,\ldots,u_N;\tau)
  &=
  -\frac{e^\frac{2\pi imu^2}{2mN\tau-1}}{\sqrt{-2mN\tau+1}}\widetilde{\mu}_N\left(\frac{u_0}{2mN\tau-1},\ldots,\frac{u_N}{2mN\tau-1};-\frac{\tau}{2mN\tau-1}\right),
\end{align}
where the definition of $\widetilde{\mu}_N$ is the same as \eqref{eq: modular completion of muN}. 
\end{coro}
%\begin{proof}[{\bf{Proof}}]
%\eqref{eq: even trans 1}: It is trivial from \eqref{eq: MN completion 4}. \\
%\eqref{eq: even trans 2}: From the similar calculation of \eqref{eq: modular STS}, we have
%\begin{align}
%&\widetilde{M}_N(u_0,\ldots,u_N;\tau)\nonumber\\
%\label{eq: modular STS 3}
%&=-e^\frac{\pi imN}{4}\frac{e^\frac{\pi imu^2}{N(m\tau-1)}}{N\sqrt{-m\tau+1}}\left[\sum_{l=0}^{N-1}(-1)^{ml}\zeta_{2N}^{ml^2-2l(j+k)}\right]_{j,k=0}^{N-1}\widetilde{M}_N\left(\frac{u_0}{m\tau-1},\ldots,\frac{u_N}{m\tau-1};-\frac{\tau}{m\tau-1}\right). 
%\end{align}
%Then, rewriting $m\mapsto 2mN$, the $(0,k)$ entry of the matrix of \eqref{eq: modular STS 3} is
%\begin{align*}
%\sum_{l=0}^{N-1}(-1)^{ml}\zeta_{2N}^{ml^2-2lk}=\sum_{l=0}^{N-1}\zeta_N^{-lk}=\begin{cases}
%N & k=0\\
%0 & k=1,\ldots,N-1
%\end{cases}.
%\end{align*}
%Therefore, we obtain the desired result. 
%\end{proof}
\begin{coro}
\label{coro: even 2}
For $m\in\Zz_{>0}$, the function $\mu_N$ satisfies the following relations: 
\begin{align}
\label{eq: even mun 1}
  \mu_N(u_0,\ldots,u_N;\tau)
  &=
  e^\frac{\pi imN}{4}\mu_N(u_0,\ldots,u_N;\tau+m),
  \\
  \mu_N(u_0,\ldots,u_N;\tau)
  &=
  -\frac{e^\frac{2\pi imu^2}{2mN\tau-1}}{\sqrt{-2mN\tau+1}}\mu_N\left(\frac{u_0}{2mN\tau-1},\ldots,\frac{u_N}{2mN\tau-1};-\frac{\tau}{2mN\tau-1}\right)
  \nonumber \\
\label{eq: even mun 2}
  &\quad +
  \frac{i^{-\frac{1}{2}-N}}{2\sqrt{2m}}\sum_{j=1}^{2m}e^\frac{\pi i(2j-1)^2}{8m}h\left(u+\frac{2j-1}{4m}-\frac{1}{2};N\tau-\frac{1}{2m}\right).
\end{align}
\end{coro}

\appendix
\def\thesection{Appendix}
\renewcommand{\theequation}{A.\arabic{equation}}
\setcounter{equation}{0}
\section{
A connection problem of the $q$-difference equation \eqref{multi gmu eq}
}

In this Appendix, we present a connection problem of the higher order $q$-difference equation \eqref{multi gmu eq}. 
\label{sec: 2}
%For a formal series $g(x)=\sum A_nx^n$, we define $N$-th order $q$-Borel transformation $\mathcal{B}_q^N$ and $q$-Laplace transformation $\mathcal{L}^N_q$ as  
%\begin{align*}
%\mathcal{B}^N_q(g)(x)&:=\sum A_nq^\frac{Nn(n-1)}{2}x^n=\mathcal{B}_{q^N}(g)(x), \\
%\mathcal{L}^N_q(g)(\lambda_0,\ldots ,\lambda_N)&:=\sum_{n\in\Zz^N}g(\lambda_0q^{|n|})\prod_{j=1}^N\frac{(\lambda_{j-1}/\lambda_j)^{n_j}q^\frac{n_j(n_j-1)}{2}}{\theta_q(\lambda_{j-1}/\lambda_j)}.
%\end{align*}
%If $N=1$ and $\lambda_1=x$, then the transformations $\mathcal{B}^N_q$ and $\mathcal{L}^N_q$ are coincides with the conventional definitions \eqref{q-Borel Laplace transformation}. 
%Also, from a simple calculation, we have $\mathcal{L}_q^N\circ\mathcal{B}_q^N(x^m T_x^ng)(\lambda_0,\ldots ,\lambda_N)=\lambda_N^m T_{\lambda_N}^n\mathcal{L}_q^N\circ\mathcal{B}_q^N(g)(\lambda_0,\ldots,\lambda_N)$. 
%Therefore, the $q$-difference equation \eqref{multi gmu eq} has convergent solutions $\mathcal{L}_q^{N-M}\circ\mathcal{B}_q^{N-M}(\tilde{f}_k)(\lambda_0,\ldots,\lambda_{N-M})$ with $\lambda_{N-M}=x$. 
\begin{lem}
\label{lem: qK CP}
The relation between the solutions $\mathcal{L}_q^{N-M}\circ\mathcal{B}_q^{N-M}(\tilde{f}_k)(\lambda_0,\ldots,\lambda_{N-M})$ and \eqref{conv solution} of the $q$-difference equation \eqref{multi gmu eq} are as follows: 
\begin{align}
&\mathcal{L}_q^{N-M}\circ\mathcal{B}_q^{N-M}(\tilde{f}_k)(\lambda_0,\ldots,\lambda_{N-M}) \nonumber
    \\
    &=\sum_{j=0}^N\frac{(\{b_sq/a_j\}_{0\leq s\leq M}^k)_\infty(\{a_r/b_k\}_{0\leq r\leq N}^j)_\infty}{(\{b_sq/b_k\}_{0\leq s\leq M}^k)_\infty(\{a_r/a_j\}_{0\leq r\leq N}^j)_\infty}\frac{\theta_q(-a_jb_k^{N-M-1}\lambda_0)}{\theta_q(-b_k^{N-M}\lambda_0)}\left[\prod_{l=1}^{N-M}\frac{\theta_q(a_j\lambda_l/b_k\lambda_{l-1})}{\theta_q(\lambda_l/\lambda_{l-1})}\right] \nonumber
    \\
    &\times{}_{N+1}\phi_N\hyper{\{a_j/b_s\}_{0\leq s\leq M},0,\ldots ,0}{\{a_jq/a_r\}_{0\leq r\leq N}^j}{q}{(-1)^{N-M}\frac{b_0\cdots b_Mq^{N+1}}{a_0\cdots a_N\lambda_{N-M}}}.
\end{align}
In particular, for $k=0$ and $b_0=1$, we have
\begin{align}
\label{eq: tphi qB}
   \mathcal{L}^{N-M}_q\circ\mathcal{B}^{N-M}_q(\tilde{f}_0)(\lambda_0,\ldots,\lambda_{N-M})
   =
   {}_N\tilde{\phi}_M\hyper{a_0,\ldots,a_N}{b_1,\ldots, b_M}{q}{\lambda_0,,-\frac{\lambda_1}{\lambda_0}\ldots,-\frac{\lambda_{N-M}}{\lambda_{N-M-1}}}, 
\end{align}
where 
\begin{align*}
{}_N\tilde{\phi}_M\hyper{a_0,\ldots,a_N}{b_1,\ldots,b_M}{q}{x_0,\ldots,x_{N-M}}
    &:=
    \sum_{j=0}^N
       \frac{(\{b_sq/a_j\}_{1\leq s\leq M})_\infty(\{a_r\}_{0\leq r\leq N}^j)_\infty}{(\{b_sq\}_{1\leq s\leq M})_\infty(\{a_r/a_j\}_{0\leq r\leq N}^j)_\infty}
       \left[\prod_{l=0}^{N-M}
          \frac{\theta_q(-a_jx_l)}{\theta_q(-x_l)}\right]
    \\
    &\quad\times
    \sum_{m=0}^\infty
       \frac{(a_j)_m(\{a_j/b_s\}_{1\leq s\leq M})_m}{(\{a_jq/a_r\}_{0\leq r\leq N})_m}
       \left(\frac{b_1\cdots b_Mq^{N+1}}{a_0\cdots a_Nx_0\cdots x_{N-M}}\right)^m. 
\end{align*}
\end{lem}
\begin{proof}[{\bf{Proof}}]
From the Watson's connection formula \cite[p. 121, (4.5.2)]{GR}: 
\begin{align*}
{}_{N+1}\phi_N\hyper{a_0,\ldots,a_N}{b_1,\ldots,b_N}{q}{x}&=\sum_{j=0}^N\frac{(\{b_k/a_j\}_{1\leq k\leq N})_\infty(\{a_k\}_{0\leq k\leq N}^j)_\infty}{(\{b_k\}_{1\leq k\leq N})_\infty (\{a_k/a_j\}_{0\leq k\leq N}^j)_\infty}\\
&\times \frac{\theta_q(-a_jx)}{\theta_q(-x)}{}_{N+1}\phi_N\hyper{a_j,a_jq/b_1,\ldots,a_jq/b_N}{\{a_jq/a_k\}_{0\leq k\leq N}^j}{q}{\frac{qb_1\cdots b_N}{a_0\cdots a_Nx}}, 
\end{align*}
we have 
\begin{align*}
&\mathcal{L}_q^{N-M}\circ\mathcal{B}_q^{N-M}(\tilde{f}_k)(\lambda_0,\ldots,\lambda_{N-M})
    \\
    &=\mathcal{L}_q^{N-M}\left({}_{N+1}\phi_N\hyper{\{a_r/b_k\}_{0\leq r\leq N}}{\{b_sq/b_k\}_{0\leq s\leq M}^k,0,\ldots, 0}{q}{\left(\frac{b_k}{q}\right)^{N-M}x}\right)
    \\
    &=
    \sum_{j=0}^N
        \frac{(\{b_sq/a_j\}_{0\leq s\leq M}^k)_\infty(\{a_r/b_k\}_{0\leq r\leq N}^j)_\infty}{(\{b_sq/b_k\}_{0\leq s\leq M}^k)_\infty(\{a_r/a_j\}_{0\leq r\leq N}^j)_\infty}\left(\frac{a_j}{b_k}\right)^{N-M}
    \\
    &\times\mathcal{L}_q^{N-M}\left(\frac{\theta_q(-a_jb_k^{N-M-1}x)}{\theta_q(-b_k^{N-M}x)}{}_{M+1}\phi_N\hyper{\{a_j/b_s\}_{0\leq s\leq M}}{\{a_jq/a_r\}_{0\leq r\leq N}^j}{q}{\left(\frac{a_jq}{b_k}\right)^{N-M}\frac{b_0\cdots b_M q^{N+1}}{a_0\cdots a_N x}}\right).
\end{align*}
Therefore, we obtain the conclusion from a bit calculation. 
\end{proof}

%Putting
%\begin{align*}
%{}_N\tilde{\phi}_M\hyper{a_0,\ldots,a_N}{b_1,\ldots,b_M}{q}{x_0,\ldots,x_{N-M}}
%    &:=
%    \sum_{j=0}^N
%       \frac{(\{b_sq/a_j\}_{1\leq s\leq M})_\infty(\{a_r\}_{0\leq r\leq N}^j)_\infty}{(\{b_sq\}_{1\leq s\leq M})_\infty(\{a_r/a_j\}_{0\leq r\leq N}^j)_\infty}
%       \left[\prod_{l=0}^{N-M}
%          \frac{\theta_q(-a_jx_l)}{\theta_q(-x_l)}\right]
%    \\
%    &\quad\times
%    \sum_{m=0}^\infty
%       \frac{(a_j)_m(\{a_j/b_s\}_{1\leq s\leq M})_m}{(\{a_jq/a_r\}_{0\leq r\leq N})_m}
%       \left(\frac{b_1\cdots b_Mq^{N+1}}{a_0\cdots a_Nx_0\cdots x_{N-M}}\right)^m. 
%\end{align*}
%Then the relation between the functions ${}_N\tilde{\phi}_M$ and $\mathcal{L}^{N-M}\circ\mathcal{B}^{N-M}(\tilde{f}_0)$ is 
%\begin{align*}
%{}_N\tilde{\phi}_M\hyper{a_0,\ldots,a_N}{b_1,\ldots, b_M}{q}{\lambda_0,,-\frac{\lambda_1}{\lambda_0}\ldots,-\frac{\lambda_{N-M}}{\lambda_{N-M-1}}}
%    =\mathcal{L}^{N-M}\circ\mathcal{B}^{N-M}(\tilde{f}_0)(\lambda_0,\ldots,\lambda_{N-M})
%\end{align*}
%and $b_0=1$. 
\begin{rmk}
The case of $M=0$ and $N=1$ with Lemma $\ref{lem: qK FS}$ and Lemma $\ref{lem: qK CP}$ is in \cite[Theorem $2.2.1$]{Zh}. 
 Also, Adachi \cite[Theorem $3.1$]{Ad} obtained a similar connection formula for the image $\mathcal{L}_{q^{N-M}}\circ\mathcal{B}_{q^{N-M}}$. 
\end{rmk}
The function ${}_N\tilde{\phi}_M$ satisfies the following properties. 
\begin{prop}
\label{prop: appendix 1}
\ \\
$(1)$\ The function ${}_N\tilde{\phi}_M$ is a symmetric function with respect to $\{x_0,\ldots,x_{N-M}\}$, $\{a_0,\ldots,a_N\}$ or $\{b_1,\ldots,b_M\}$. \\
$(2)$\ The function ${}_N\tilde{\phi}_M$ satisfies the following $q$-difference equation: 
\begin{align*}
\left[T_x^{N-M}(1-T_x)(1-b_1T_x)\cdots (1-b_MT_x)-(-1)^{N-M}X_{N-M}(1-a_0T_x)\cdots (1-a_NT_x)\right]=0.
\end{align*}
$(3)$\ The function ${}_N\tilde{\phi}_M$ satisfies the following relation: 
\begin{align}
&{}_N\tilde{\phi}_M\hyper{a_0,\ldots, a_N}{b_1,\ldots, b_M}{q}{x_0,\ldots, x_{N-M}}-{}_N\tilde{\phi}_M\hyper{a_0,\ldots, a_N}{b_1,\ldots, b_M}{q}{x_0y^{-1},x_1y,x_2,\ldots, x_{N-M}}
    \nonumber\\
    &=\frac{\theta_q(-y)\theta_q(-x_1y/x_0)}{\theta_q(-y/x_0)\theta_q(-x_1y)\theta_q(-x_0)\cdots\theta_q(-x_{N-M})}\sum_{j=0}^N\frac{(\{b_sq/a_j\}_{1\leq s\leq M})_\infty(\{a_r\}_{0\leq r\leq N}^j)_\infty}{(\{b_sq\}_{1\leq s\leq M})_\infty(\{a_r/a_j\}_{0\leq r\leq N}^j)_\infty}
    \nonumber\\
\label{tildephi}
    &\times
    \theta_q(-a_j)\theta_q(-a_jx_0x_1)\left[\prod_{l=2}^{N-M}\theta_q(-a_jx_l)\right]\sum_{m=0}^\infty\frac{(a_j)_m(\{a_j/b_s\}_{1\leq s\leq M})_m}{(\{a_jq/a_r\}_{0\leq r\leq N})_m}\left(\frac{b_1\cdots b_Mq^{N+1}}{a_0\cdots a_Nx_0\cdots x_{N-M}}\right)^m.
\end{align}
\end{prop}
\begin{proof}[{\bf{Proof}}]
From the definition and $q$-Borel summable construction \eqref{eq: tphi qB} of the ${}_N\tilde{\phi}_M$, $(1)$ and $(2)$ are clear. Also, from a Riemann's relation of the theta function (for example, see \cite{M}): 
\begin{align*}
\frac{\theta_q(-ax_0)\theta_q(-ax_1)}{\theta_q(-x_0)\theta_q(-x_1)}-\frac{\theta_q(-ax_0/y)\theta_q(-ax_1y)}{\theta_(-x_0/y)\theta_q(-x_1y)}=\frac{\theta_q(-a)\theta_q(-y)\theta_q(-ax_0x_1)\theta_q(-x_1y/x_0)}{\theta_q(-x_0)\theta_q(x_1)\theta_q(-y/x_0)\theta_q(-x_1y)},
\end{align*}
we obtain the conclusion of $(3)$. 
\end{proof}
\begin{rmk}
Proposition $\ref{prop: appendix 1}$ $(1)$, $(2)$ and $(3)$ are regarded as a generalization of \eqref{eq: fN 4}, \eqref{eq: fN 7} and \eqref{eq: fN 2}, respectively. 
However, we can not prove \eqref{eq: fN 2}, \eqref{eq: fN 4} and \eqref{eq: fN 7} directly from Proposition $\ref{prop: appendix 1}$. 
\end{rmk}

\section*{Acknowledgments}
We are grateful to Professor Yasuhiko Yamada (Kobe University) for his helpful comments of modular properties of the multivariable $\mu$-function. 
This work was supported by JSPS KAKENHI Grant Number 21K13808 and JST CREST Grant Number JP19209317.
%\section{A derivation of the multivariable generlized $\mu$-function}

\end{document}